\UseRawInputEncoding
\documentclass[11pt]{amsart}%
\usepackage{palatino, mathpazo}
\usepackage{amsfonts}
\usepackage{amsmath}
\usepackage{amssymb,latexsym}
\usepackage{graphicx}
\usepackage[mathscr]{eucal}
\usepackage{amssymb}%
\usepackage{arydshln} 
\usepackage{booktabs}
\usepackage{multirow}
\usepackage{cases}
\usepackage{color}
\usepackage{hyperref}
\usepackage[all]{xy}
\usepackage{tikz,mathpazo}
\providecommand{\U}[1]{\protect \rule{.1in}{.1in}}

\newtheorem{theorem}{Theorem}[section]

\newtheorem{corollary}[theorem]{Corollary}

\newtheorem{lemma}[theorem]{Lemma}

\newtheorem{proposition}[theorem]{Proposition}
\newtheorem{remark}[theorem]{Remark}

\numberwithin{equation}{section}
\newcommand{\R}{{\mathbb{R}}}
\newcommand{\N}{{\mathbb{N}}}
\newcommand{\Z}{{\mathbb{Z}}}
\newcommand{\C}{{\mathbb{C}}}
\renewcommand{\a}{{\mathbf{a}}}

\renewcommand{\l}{{\lambda}}

\begin{document}
\title[Spectrum of the Lam\'{e} operator]
{Spectrum of the Lam\'{e} operator along $\mathrm{Re}\tau={1}/{2}:$ The genus $3$ case}
\author{Erjuan Fu}
\address{Yanqi Lake Beijing Institute of Mathematical Sciences and Applications,
 Beijing, 101408, China}
\email{fej.2010@tsinghua.org.cn; ejfu@bimsa.cn}

\maketitle

\begin{abstract}
In this paper, we study the spectrum $\sigma(L)$ of the Lam\'{e} operator
\[L=\frac{d^2}{dx^2}-12\wp(x+z_0;\tau)\quad \text{in}\;\;L^2(\mathbb{R}, \mathbb{C}),\]
where $\wp(z;\tau)$ is the Weierstrass elliptic function with periods $1$ and $\tau$, and $z_0\in\mathbb{C}$ is chosen such that $L$ has no singularities on $\mathbb{R}$.
 We prove that a point $\lambda\in \sigma(L)$ is an intersection point of different spectral arcs but not a zero of the spectral polynomial if and only if $\lambda$ is a zero of the following cubic polynomial: 
\begin{equation*}
    \frac{4}{15} \l^3+\frac{8}{5}\eta_1 \l^2-3g_2 \l+9g_3-6\eta_1 g_2=0.
\end{equation*}
We also study the deformation of the spectrum as $\tau=\frac12+ib$ with $b>0$ varying. We discover $10$ different types of graphs for the spectrum as $b$ varies around the double zeros of the spectral polynomial.
\end{abstract}

\tableofcontents

\section{Introduction}
Let $T_{\tau}:=\mathbb{C}/\Lambda_\tau$ be a flat torus with $\Lambda_\tau=\mathbb{Z}+\mathbb{Z}\tau$ and $\tau \in \mathbb{H}=\{  \tau\in\C|\operatorname{Im}\tau>0\}$.
Recall that 
\begin{equation*}
\label{40-3}\wp(z;\tau):=\frac{1}{z^{2}}+\sum_{(m,n)\in \Z^2\setminus\{(0,0)\} }\left(  \frac{1}{(z-m-n\tau)^{2}}-\frac
{1}{(m+n\tau)^{2}}\right)
\end{equation*}
is the
Weierstrass elliptic function with basic periods $\omega_{1}=1$ and $\omega_{2}
=\tau$. Denote by $\omega_3=1+\tau$. It is well known that
\begin{equation}\label{eq-wp}
\wp'(z;\tau)^2
=4\prod_{k=1}^{3}(\wp(z;\tau)-e_{k}(\tau))
=4\wp(z;\tau)^3-g_{2}(\tau)\wp(z;\tau)-g_{3}(\tau),\end{equation}
where $e_k(\tau)=\wp(\frac{\omega_k}{2}; \tau)$, $k=1,2,3$, and $g_2(\tau), g_3(\tau)$ are well-known invariants of the elliptic curve.

The Weierstrass
zeta function is defined by  $\zeta(z)=\zeta(z;\tau):=-\int^{z}\wp(\xi;\tau)d\xi$   with two quasi-periods $\eta_{j}=\eta_{j}(\tau)$, $j=1,2$:%
\begin{equation*}
\eta_{j}(\tau)=2\zeta(\tfrac{\omega_{j}}{2};\tau)=\zeta(z+\omega_{j}%
;\tau)-\zeta(z;\tau),\quad j=1,2,\label{40-2}%
\end{equation*}
and the Weierstrass sigma function is defined by $\sigma(z)=\sigma(z;\tau):=\exp \int^{z}\zeta(\xi)d\xi$. Notice that $\zeta(z)$ is an odd meromorphic function with simple poles at $\mathbb{Z}+\mathbb{Z}\tau$ and $\sigma(z)$ is an odd
entire function with simple zeros at $\mathbb{Z}+\mathbb{Z}\tau$. 

In this paper, we study
the spectrum $\sigma(L^{g}_\tau)$ of the classical Lam\'{e} operator \cite{Lame}
\begin{equation*}
L_\tau^g=\frac{d^2}{dx^2}-g(g+1)\wp(x+z_0;\tau),\quad  x\in\mathbb{R}\end{equation*}
in $L^2(\R, \C)$,
where  $g\in \N$ and $z_0\in\mathbb{C}$ is chosen such that $\wp(x+z_0;\tau)$ has no singularities on $\mathbb{R}$.
 Remark that $\sigma(L_\tau^g)$ does not depend on the choice of $z_0$ due to the fact that the Lam\'{e} potential $-g(g+1)\wp(z;\tau)$ is a Picard potential in the sense of Gesztesy and Weikard \cite{GW} (i.e. all solutions of the Lam\'{e} equation
\begin{equation*}
y''(z)= \left(g(g+1)\wp(z;\tau)+\l\right)y(z),\quad z\in \mathbb{C}
 \end{equation*}
are meromorphic in $\mathbb{C}$).

The spectral theory of the Schr\"{o}dinger operator $L$ with complex periodic smooth potentials  has attracted significant attention and has been studied widely in the literature; see e.g. \cite{BG-JAM,B-CPAM,GW,GW2,HHV,Rofe-Beketov} and references therein.
In this theory, it is known \cite{Rofe-Beketov} that
\[\sigma(L)=\Delta^{-1}([-2,2])=\{\l\in\mathbb{C}\,|\, -2\leq\Delta(\l)\leq 2\},\]
where $\Delta(\l)$ denotes the Hill's discriminant  which is the trace of the monodromy matrix of $Ly=\l y$ with respect to $x\to x+1$.
Furthermore, it was proved in \cite{GW} that $\sigma(L)$ consists of finitely many analytic arcs if the potential of $L$ is a Picard potential. In this paper, as in \cite{CFL-ADV, GW} we call the arcs of $\sigma(L)$ as \emph{spectral arcs}.

For the simplest case $\mathrm{Re}\tau=0$, since the Lam\'e potential $-g(g+1)\wp(x+\frac{\omega_2}{2}; \tau)$ is real-valued and smooth on $\R$,
Ince \cite{Ince} proved that there are $2g+1$ distinct real numbers $\l_0>\l_1> \cdots> \l_{2g}$ such that the spectrum
\begin{equation}\label{spectrum-Rline}
    \sigma(L_\tau^g)=(-\infty, \l_{2g}]\cup[\l_{2g-1}, \l_{2g-2}]\cup\cdots\cup[\l_1, \l_0] \subseteq \R.
\end{equation}

However, the spectrum $\sigma(L_\tau^g)$ is no longer of the form (\ref{spectrum-Rline}) for general $\tau$'s and becomes very complicated; see e.g. \cite{BG-JAM,  CFL-AJM, CL-PAMS, GW-Lame, GW2, HHV} and references therein. Gesztesy-Weikard \cite{GW2} and Haese-Hill et al. \cite{HHV} concentrated on the $g=1$ case, for which the spectrum $\sigma(L_\tau^1)$ consists of two regular analytic arcs and so there are totally three different types of graphs for different $\tau$'s as shown in the following figure (see also \cite{ BG-JAM, GW-Lame}).

\medskip

\begin{center}
\begin{tabular}{ccccc}
\hline\\
\begin{tikzpicture}
\draw[thick,smooth](0,0.25)..controls (0.5,0.5) and (0.75,0.5)..(1,0);
\draw[thick,smooth](0.75,0.75)..controls (0.75,0.4) and (1,0.4)..(1.25,0.25);
\fill(1,0) circle(1.5pt);
\fill(0.75,0.75) circle(1.5pt);
\fill(1.25,0.25) circle(1.5pt);
\end{tikzpicture}   &\text{\qquad}&
\begin{tikzpicture}
\draw[thick,smooth](0.25,0.75)..controls (0.45,0.4) and (0.6,0.4)..(0.8,0.5);
\draw[thick,smooth](1,0.75)..controls (0.75,0.6) and (0.75,0.4)..(1,0);
\fill(0.8,0.5) circle(1.5pt);
\fill(1,0.75) circle(1.5pt);
\fill(1,0) circle(1.5pt);
\end{tikzpicture}
&\text{\qquad} &
\begin{tikzpicture}
\draw[thick,smooth](0.25,0.6)..controls (0.5,0.4) and (0.75,0.4)..(1,1);
\draw[thick,smooth](0.5,0.75)..controls (0.75,0.75) and (1,0.75)..(1.25,0.25);
\fill(0.5,0.75) circle(1.5pt);
\fill(1,1) circle(1.5pt);
\fill(1.25,0.25) circle(1.5pt);
\fill(0.85,0.7) circle(1.5pt);
\end{tikzpicture} \\
  \scriptsize{(a)}  && \scriptsize{(b)}  && \scriptsize{(c)} \\
  \hline \\
  && Figure 1 &&
\end{tabular}
\end{center}
\medskip


In particular, different spectral arcs might intersect as shown in Figure (b)-(c).
It was pointed out in \cite[Section 5]{HHV} that the rigorous analysis of $g\geq 2$ cases seems to be difficult since the related explicit formulae quickly become quite complicated as $g$ grows.
Recently,
\cite[Lemma 4.1]{CFL-ADV} inferred that
 the spectrum $\sigma(L_\tau^g)$ is symmetric with respect to the real line $\R$  if $\tau=\frac{1}{2}+bi$ with $b>0$. Furthermore, it was proved in \cite{CFL-ADV} that   the spectrum  $\sigma(L_{\frac{1}{2}+bi}^2)$ in $g=2$ case has exactly 9 different types of graphs for different $b$'s, and
the continuous deformation of the spectrum
as $b$ increases is as follows.

\medskip
\begin{center}
    \begin{tabular}{ccccc}
    \hline\\
 {\begin{tikzpicture}
\draw[thick] (-0.25,0)--(0.6,0);
\draw[thick]  (1.1,0)--(1.6,0);
   \draw [thick,domain=-30:30] plot ({cos(\x)-0.1}, {sin(\x)});
   \fill(0.6,0) circle (1.5pt) ;
         \fill ({cos(30)-0.1},{sin(30)}) circle (1.5pt);
    \fill ({cos(30)-0.1},{-sin(30)}) circle (1.5pt) ;
    \fill (1.1,0) circle (1.5pt);
      \fill (1.6,0) circle (1.5pt);
      \draw(-0.25, 0.1) node[below]{\tiny{$-\infty$}};
\end{tikzpicture}}&&
  {\begin{tikzpicture}
\draw[thick] (-0.25,0)--(0.75,0);
\draw[thick]  (1.1,0)--(1.6,0);
   \draw [thick,domain=-30:30] plot ({cos(\x)-0.25}, {sin(\x)});
   \fill(0.75,0) circle (1.5pt) ;
         \fill ({cos(30)-0.25},{sin(30)}) circle (1.5pt);
    \fill ({cos(30)-0.25},{-sin(30)}) circle (1.5pt) ;
    \fill (1.1,0) circle (1.5pt);
      \fill (1.6,0) circle (1.5pt);
      \draw(-0.25, 0.1) node[below]{\tiny{$-\infty$}};
\end{tikzpicture}}
&&
  {\begin{tikzpicture}
\draw[thick] (-0.25,0)--(0.75,0);
\draw[thick]  (1.1,0)--(1.6,0);
   \draw [thick,domain=-30:30] plot ({cos(\x)-0.5}, {sin(\x)});
   \fill(0.75,0) circle (1.5pt) ;
         \fill ({cos(30)-0.5},{sin(30)}) circle (1.5pt);
    \fill ({cos(30)-0.5},{-sin(30)}) circle (1.5pt) ;
    \fill (1.1,0) circle (1.5pt);
      \fill (1.6,0) circle (1.5pt);
      \fill (0.5,0) circle (1.5pt);
      \draw(-0.25, 0.1) node[below]{\tiny{$-\infty$}};
\end{tikzpicture}}
       \\ \scriptsize{(1)} &&\scriptsize{(2)}&&\scriptsize{(3)} \\ \hline\\
       {\begin{tikzpicture}
\draw[thick] (-0.25,0)--(1.6,0);
 \draw [thick,domain=-30:30] plot ({cos(\x))-0.5}, {sin(\x)});
         \fill ({cos(30))-0.5},{sin(30)}) circle (1.5pt);
    \fill ({cos(30)-0.5},{-sin(30)}) circle (1.5pt) ;
        \fill (0.5,0) circle (1.5pt);
      \fill (1.6,0) circle (1.5pt);
      \fill (1,0) circle (1.5pt)
      node[below]{\tiny{0}};
      \draw(-0.25, 0.1) node[below]{\tiny{$-\infty$}};
\end{tikzpicture}}
&&
  \begin{tikzpicture}
\draw[thick] (-1.5,0)--(0,0);
   \draw [thick,domain=150:210] plot ({cos(\x)+0.5}, {sin(\x)});
      \draw [thick,domain=-30:30] plot ({cos(\x)-1.75}, {sin(\x)});
 \fill ({cos(30)-1.75},{sin(30)}) circle (1.5pt);
   \fill ({cos(30)-1.75},{-sin(30)}) circle (1.5pt);
         \fill ({cos(150)+0.5},{sin(150)}) circle (1.5pt) ;
            \fill ({cos(150)+0.5},{-sin(150)}) circle (1.5pt) ;
   \fill (0,0) circle (1.5pt);
      \fill(-0.5,0) circle (1.5pt);
         \fill (-0.75,0) circle (1.5pt);
       \draw(-1.5, 0.1) node[below]{\tiny{$-\infty$}};
\end{tikzpicture}
&&
\begin{tikzpicture}
\draw[thick] (-1.5,0)--(0,0);
    \draw [thick,domain=-170:-10] plot ({0.25*cos(\x)-0.5}, {0.5*sin(\x)+0.5});
    \draw [thick,domain=10:170] plot ({0.25*cos(\x)-0.5}, {0.5*sin(\x)-0.5});
   \fill ({0.25*cos(10)-0.5},{-0.5*sin(10)+0.5}) circle (1.5pt);
   \fill ({0.25*cos(170)-0.5},{-0.5*sin(170)+0.5}) circle (1.5pt);
   \fill ({0.25*cos(10)-0.5},{0.5*sin(10)-0.5}) circle (1.5pt) ;
      \fill ({0.25*cos(170)-0.5},{0.5*sin(170)-0.5}) circle (1.5pt) ;
   \fill (0,0) circle(1.5pt) ;
      \fill (-0.5,0) circle (1.5pt);
     \draw(-1.5, 0.1) node[below]{\tiny{$-\infty$}};
\end{tikzpicture}
    \\ \scriptsize{(4)} &&\scriptsize{(5)}&&\scriptsize{(6)} \\ \hline\\
{\begin{tikzpicture}
\draw[thick] (-0.25,0)--(1,0);
   \draw [thick,domain=60:120] plot ({cos(\x)+0.75}, {sin(\x)-1.2});
   \draw [thick,domain=-120:-60] plot ({cos(\x)+0.75}, {sin(\x)+1.2});
   \fill(1,0) circle (1.5pt) ;
         \fill ({cos(60)+0.75},{sin(60)-1.2}) circle (1.5pt);
    \fill ({cos(120)+0.75},{sin(120)-1.2}) circle (1.5pt) ;
         \fill ({cos(60)+0.75},{-sin(60)+1.2}) circle (1.5pt);
    \fill ({cos(120)+0.75},{-sin(120)+1.2}) circle (1.5pt) ;
     \draw(-0.25, 0.1) node[below]{\tiny{$-\infty$}};
\end{tikzpicture}}
&&
{\begin{tikzpicture}
\draw[thick] (-0.25,0)--(0.75,0);
   \draw [thick,domain=150:210] plot ({cos(\x)+2.35}, {sin(\x)});
   \fill(0.75,0) circle (1.5pt) ;
         \fill ({cos(150)+2.35},{sin(150)}) circle (1.5pt);
    \fill ({cos(210)+2.35},{-sin(150)}) circle (1.5pt) ;
    \fill (1.35,0) circle (1.5pt)
    node[right]{\tiny{0}};
     \draw(-0.25, 0.1) node[below]{\tiny{$-\infty$}};
\end{tikzpicture}}
&&
{\begin{tikzpicture}
\draw[thick] (-0.25,0)--(0.75,0);
\draw[thick]  (1.1,0)--(1.6,0);
   \draw [thick,domain=150:210] plot ({cos(\x)+2.35}, {sin(\x)});
   \fill(0.75,0) circle (1.5pt) ;
         \fill ({cos(150)+2.35},{sin(150)}) circle (1.5pt);
    \fill ({cos(210)+2.35},{-sin(150)}) circle (1.5pt) ;
    \fill (1.1,0) circle (1.5pt);
      \fill (1.6,0) circle (1.5pt);
      \fill (1.35,0) circle (1.5pt);
      \draw(-0.25, 0.1) node[below]{\tiny{$-\infty$}};
\end{tikzpicture}}
       \\ \scriptsize{(7)} &&\scriptsize{(8)}&&\scriptsize{(9)}\\ \hline \\
       && Figure 2 &&
    \end{tabular}
   \end{center}
\medskip



In this paper, we will focus on the genus  
$g=3$ case. To the best of our knowledge, there seems no more explicit description of $\sigma(L_\tau^3)$ in the literature.
It is well known (see \cite[p.569]{LW2}) that the associated spectral polynomial $Q_3(\l;\tau)$ is given by
\begin{equation}\label{spectralpolynomial}
Q_3(\l;\tau)
=\l\prod_{k=1}^3\left(\l^2-6e_k\l
+15\left(3e_k^2-g_2\right)\right),
\end{equation}
 so the associated hyperelliptic curve $\{(\lambda, C)\,|\,C^2=Q_3(\l;\tau)\}$ is of genus $3$.

By applying \cite[Theorem 4.1]{GW}, we see that \emph{the spectrum $\sigma(L_{\tau}^3)$ consists of
$\tilde{g}\leq 3$ bounded simple analytic arcs $\sigma_{k}$ and {\it one} semi-infinite simple analytic arc $\sigma_{\infty}$ which tends
to $-\infty+\langle q\rangle$, with $\langle q\rangle=\int
_{x_{0}}^{x_{0}+1}q(x)dx$, i.e.
\begin{equation*}
\sigma(L)=\sigma_{\infty}\cup \cup_{k=1}^{\tilde{g}}\sigma_{k},\quad \tilde{g}\leq 3,
\end{equation*}
where the finite endpoints of such arcs must be zeros of the spectral polynomial $Q_3(\l;\tau)$ with odd order.}

As shown in the above figures, there are two kinds of intersection points of spectral arcs: One is that this intersection point is also an endpoint of some arc (see Figures (b) and (2)), so it is met by $2k+1$ semi-arcs for some $k\geq 1$; the other is that this intersection point is not an endpoint (see Figures (c) and (3)-(6), (9)), so it is met by $2k$ semi-arcs for some $k\geq 2$, i.e. it is an inner point of $k$ arcs; and we call it  \emph{an inner intersection point}.

First, we have to determine completely the inner intersection points of different arcs in order to study the geometry of $\sigma(L_\tau^3)$. Our first result is as follows.

\begin{theorem}\label{thm-inner}
Let $\tau\in \mathbb{H}$ and $\l_0\in \sigma(L_\tau^3)$ with $Q_3(\l_0; \tau)\neq 0$. Then $\l_0$ is an inner intersection point if and only if  
$\l_0$ satisfies the following cubic equation
\begin{equation*}
f(\l):=\frac{4}{15} \l^3+\frac{8}{5}\eta_1 \l^2-3g_2 \l+9g_3-6\eta_1 g_2=0.
\end{equation*}
\end{theorem}

Now we study the deformation of the spectrum $\sigma(L_\tau^3)$ as $\tau=\frac{1}{2}+bi$ deforms. This problem is challenging and we can only obtain some partial results.
For this purpose, we need to analysis when some endpoint $\lambda_0$ (i.e. $\l_0$ is a zero of $Q_3(\l;\tau)$ with odd order) could be an intersection point.

\begin{theorem}\label{thm-spec-poly-zeros}
Let $\tau=\frac{1}{2}+bi$ with $b>0$ and $d(\l)$ denote the number of semi-arcs met at $\lambda$. Then the zeros of the spectral polynomial $Q_3(\l; \tau)$ are listed as follows
\begin{equation*}
    0, \mu, \overline{\mu}, \nu, \overline{\nu}, \vartheta_+, \vartheta_-,
\end{equation*}
where the multiplicities of $0, \mu, \overline{\mu},  \nu, \overline{\nu}$ are $1$. Furthermore, there is some $\beta\approx 1.0979\in  \left(\frac{\sqrt{3}}{2}, +\infty\right)$ such that
\begin{enumerate}
    \item[(z1)] $\vartheta_-=\vartheta_+=3e_1\in \R$ if and only if $b\in \{\beta, \widehat{\beta}\}$, where $\widehat{\beta}= \frac{1}{4\beta}\approx 0.2277$.
    \item[(z2)] $\vartheta_-, \vartheta_+\in \R$ and $\vartheta_-<3e_1< \vartheta_+<0$ if and only if  $b\in (0,\widehat{\beta})$.
       \item[(z3)] $\vartheta_-, \vartheta_+\in \R$ and $0<\vartheta_-<3e_1< \vartheta_+$ if and only if  $b\in (\beta, +\infty)$.
    \item[(z4)] $\vartheta_-=\overline{\vartheta_+}\not\in \R$ if and only if $b\in (\widehat{\beta}, \beta)$.
\end{enumerate}
Moreover,\begin{enumerate}
   \item[(d1)]$d(0)=1+2d$ for some $d\geq 1$ (i.e. the endpoint $0$ is an intersection point) if and only if $b=b_{g}\approx 0.47$;
otherwise, $d(0)=1$.
   \item[(d2)] $d(\mu)=d(\nu)=d(\overline{\mu})=d(\overline{\nu})=1$, i.e. the endpoints $\mu,\nu,\overline{\mu},\overline{\nu}$ can not be intersection points.
       \item[(d3)] \begin{equation*}d(\vartheta_-)=d(\vartheta_+)=
       \begin{cases}
       2 \qquad \text{if} \quad b\in\{\widehat{\beta}, \beta\}\\
       1 \qquad \text{if} \quad b\in (\widehat{\beta}, \beta).
       \end{cases}
       \end{equation*}
\end{enumerate}
\end{theorem}



Note that $Q_3(\l; \frac{1}{2}+bi)$  has double zeros at $b\in\{\beta,\widehat{\beta}\}$.  Here we mainly study the deformation of the spectrum $\sigma(L_\tau^3)$ along $\tau=\frac{1}{2}+bi$ with $b$ varying around $\widehat{\beta}$ and $\beta$, respectively.
  We discover $10$ different patterns for the spectrum stated in the following theorem. See also the following figure for these $10$ rough graphs.

 \begin{theorem}\label{thm-main}
 Let $\tau=\frac{1}{2}+ib$ with $b>0$ and denote $L_b^3:=L_\tau^3$. Then the spectrum $\sigma(L_b^3)$ is symmetric with respect to the real line $\R$. Furthermore,
 there exist $\varepsilon>0, \delta>0$ sufficient small  and  some $k_1\in (0, \widehat\beta)$, $\alpha\approx 0.23217$ 
such that
 the rough graphs of the spectra $\sigma(L_b^3)$  for $b\in (0, \alpha+\varepsilon)\cup (\beta-\delta, +\infty)$ are described as follows.

Here, the notations $\sigma_i$ with $i=1,2,3$ denote simple arcs symmetric with respect to $\R$ and they are disjoint with each other. The notation $\sigma_4$ denotes a simple arc in $\{z\in \mathbb{C}\,|\,\operatorname{Im} z\geq 0\}$ and $\overline{\sigma_4}$ denotes the conjugate of $\sigma_4$. The notations $\lambda_\cdot$ denote real  roots of $f$, where $\lambda_-<0$ and $\lambda_+>0$.

\medskip
\begin{center}
     {\begin{tabular}{ccccccc}
\hline\\
\begin{tikzpicture}
\draw[thick] (-1.5,0)--(-0.75,0);
\draw[thick] (-0.35,0)--(0,0);
      \draw [thick,domain=-30:30] plot ({cos(\x)-1.5}, {sin(\x)});
         \draw [thick,domain=150:210] plot ({cos(\x)+1.25}, {sin(\x)});
  \fill ({cos(30)-1.5},{sin(30)}) circle (1.5pt) ;
   \fill ({cos(30)-1.5},{-sin(30)}) circle (1.5pt) ;
      \fill ({cos(150)+1.25},{sin(150)}) circle (1.5pt) ;
            \fill ({cos(150)+1.25},{-sin(150)}) circle (1.5pt) ;
   \fill (0,0) circle (1.5pt) node[below]{\tiny{0}};
      \fill (-0.35,0) circle (1.5pt);
 \fill(-0.75,0) circle (1.5pt);
       \draw(-1.5, 0.1) node[below]{\tiny{$-\infty$}};
\end{tikzpicture}
  & &
\begin{tikzpicture}
\draw[thick] (-1.5,0)--(-0.75,0);
\draw[thick] (-0.35,0)--(0,0);
      \draw [thick,domain=-30:30] plot ({cos(\x)-1.75}, {sin(\x)});
         \draw [thick,domain=150:210] plot ({cos(\x)+1.25}, {sin(\x)});
  \fill ({cos(30)-1.75},{sin(30)}) circle (1.5pt) ;
   \fill ({cos(30)-1.75},{-sin(30)}) circle (1.5pt) ;
      \fill ({cos(150)+1.25},{sin(150)}) circle (1.5pt) ;
            \fill ({cos(150)+1.25},{-sin(150)}) circle (1.5pt) ;
   \fill (0,0) circle (1.5pt) node[below]{\tiny{0}};
      \fill (-0.35,0) circle (1.5pt);
 \fill(-0.75,0) circle (1.5pt);
       \draw(-1.5, 0.1) node[below]{\tiny{$-\infty$}};
\end{tikzpicture}
  & &\begin{tikzpicture}
\draw[thick] (-1.5,0)--(-0.75,0);
\draw[thick] (-0.35,0)--(0,0);
      \draw [thick,domain=-30:30] plot ({cos(\x)-2}, {sin(\x)});
         \draw [thick,domain=150:210] plot ({cos(\x)+1.25}, {sin(\x)});
  \fill ({cos(30)-2},{sin(30)}) circle (1.5pt) ;
   \fill ({cos(30)-2},{-sin(30)}) circle (1.5pt) ;
      \fill ({cos(150)+1.25},{sin(150)}) circle (1.5pt) ;
            \fill ({cos(150)+1.25},{-sin(150)}) circle (1.5pt) ;
   \fill (0,0) circle (1.5pt) node[below]{\tiny{0}};
      \fill (-0.35,0) circle (1.5pt);
 \fill(-0.75,0) circle (1.5pt);
      \fill (-1,0) circle (1.5pt);
       \draw(-1.5, 0.1) node[below]{\tiny{$-\infty$}};
\end{tikzpicture}&&
  \begin{tikzpicture}
\draw[thick] (-1.5,0)--(0,0);
      \draw [thick,domain=-30:30] plot ({cos(\x)-2}, {sin(\x)});
         \draw [thick,domain=150:210] plot ({cos(\x)+1.25}, {sin(\x)});
         \fill ({cos(30)-2}, {sin(30)}) circle (1.5pt);
         \fill ({cos(30)-2}, {-sin(30)}) circle (1.5pt);
         \fill ({cos(150)+1.25}, {sin(150)}) circle (1.5pt);
         \fill ({cos(150)+1.25}, {-sin(150)}) circle (1.5pt);
   \fill (0,0) circle (1.5pt) node[below]{\tiny{0}};
      \fill (-0.5,0) circle (1.5pt) node[below]{\tiny{$3e_1$}};
         \fill (-1,0) circle (1.5pt);
       \draw(-1.5, 0.1) node[below]{\tiny{$-\infty$}};
\end{tikzpicture}
  \\

\scriptsize{ $0<b<k_1$}&& \scriptsize{ $b=k_1$} && \scriptsize{ $k_1<b<\widehat\beta$}
 & &\scriptsize{$b=\widehat{\beta}$}
 \\
  \hline\\
    \begin{tikzpicture}
\draw[thick] (-1.25,0)--(0,0);
   \draw [thick,domain=150:210] plot ({cos(\x)+0.5}, {sin(\x)});
      \draw [thick,domain=-30:30] plot ({cos(\x)-1.75}, {sin(\x)});
         \draw [thick,domain=150:210] plot ({cos(\x)+1.25}, {sin(\x)});
 \fill ({cos(30)-1.75},{sin(30)}) circle (1.5pt);
   \fill ({cos(30)-1.75},{-sin(30)}) circle (1.5pt);
   \fill ({cos(150)+1.25},{sin(150)}) circle (1.5pt) ;
      \fill ({cos(150)+1.25},{-sin(150)}) circle (1.5pt) ;
         \fill ({cos(150)+0.5},{sin(150)}) circle (1.5pt) ;
            \fill ({cos(150)+0.5},{-sin(150)}) circle (1.5pt) ;
   \fill (0,0) circle (1.5pt) node[below]{\tiny{0}};
      \fill(-0.5,0) circle (1.5pt);
         \fill (-0.75,0) circle (1.5pt);
       \draw(-1.25, 0.1) node[below]{\tiny{$-\infty$}};
\end{tikzpicture}
&&
\begin{tikzpicture}
\draw[thick] (-1.25,0)--(0,0);
    \draw [thick,domain=-170:-10] plot ({0.25*cos(\x)-0.5}, {0.5*sin(\x)+0.5});
    \draw [thick,domain=10:170] plot ({0.25*cos(\x)-0.5}, {0.5*sin(\x)-0.5});
         \draw [thick,domain=150:210] plot ({cos(\x)+1.25}, {sin(\x)});
   \fill ({0.25*cos(10)-0.5},{-0.5*sin(10)+0.5}) circle (1.5pt);
   \fill ({0.25*cos(170)-0.5},{-0.5*sin(170)+0.5}) circle (1.5pt);
   \fill ({0.25*cos(10)-0.5},{0.5*sin(10)-0.5}) circle (1.5pt) ;
      \fill ({0.25*cos(170)-0.5},{0.5*sin(170)-0.5}) circle (1.5pt) ;
         \fill ({cos(210)+1.25},{sin(210)}) circle (1.5pt) ;
      \fill ({cos(210)+1.25},{-sin(210)}) circle (1.5pt) ;
   \fill (0,0) circle(1.5pt) node[below]{\tiny{0}};
      \fill (-0.5,0) circle (1.5pt);
     \draw(-1.25, 0.1) node[below]{\tiny{$-\infty$}};
\end{tikzpicture}
&&
 \begin{tikzpicture}
\draw[thick] (-1.25,0)--(0,0);
    \draw [thick,domain=-145:-45] plot ({0.5*cos(\x)-0.5}, {0.5*sin(\x)+0.6});
    \draw [thick,domain=45:145] plot ({0.5*cos(\x)-0.5}, {0.5*sin(\x)-0.6});
         \draw [thick,domain=150:210] plot ({cos(\x)+1.25}, {sin(\x)});
   \fill ({0.5*cos(45)-0.5},{-0.5*sin(45)+0.6}) circle (1.5pt);
      \fill ({0.5*cos(145)-0.5},{-0.5*sin(145)+0.6}) circle (1.5pt);
   \fill ({0.5*cos(45)-0.5},{0.5*sin(45)-0.6}) circle (1.5pt) ;
     \fill ({0.5*cos(145)-0.5},{0.5*sin(145)-0.6}) circle (1.5pt) ;
        \fill ({cos(150)+1.25},{sin(150)}) circle (1.5pt) ;
        \fill ({cos(150)+1.25},{-sin(150)}) circle (1.5pt) ;
   \fill (0,0) circle (1.5pt) node[below]{\tiny{0}};
   \draw(-1.25, 0.1) node[below]{\tiny{$-\infty$}};
\end{tikzpicture} && unknown
\\ 
\scriptsize{$\widehat{\beta}<b<\alpha$}&&\scriptsize{$b=\alpha$} && \scriptsize{$\alpha<b<\alpha+\varepsilon$}  && \scriptsize{$\alpha+\varepsilon\leq b\leq \beta-\delta$}\\
\hline\\
\begin{tikzpicture}
\draw[thick] (-1.5,0)--(0,0);
      \draw [thick,domain=-30:30] plot ({cos(\x)-1.75}, {sin(\x)});
         \draw [thick,domain=220:280] plot ({cos(\x)+0.5}, {sin(\x)+1.25});
            \draw [thick,domain=80:140] plot ({cos(\x)+0.5}, {sin(\x)-1.25});
  \fill ({cos(30)-1.75},{sin(30)}) circle (1.5pt);
    \fill ({cos(30)-1.75},{-sin(30)}) circle (1.5pt);
    \fill ({cos(220)+0.5},{sin(220)+1.25}) circle (1.5pt);
     \fill ({cos(280)+0.5},{sin(280)+1.25}) circle (1.5pt);
         \fill ({cos(80)+0.5},{sin(80)-1.25}) circle (1.5pt);
     \fill ({cos(140)+0.5},{sin(140)-1.25}) circle (1.5pt);
   \fill (0,0) circle (1.5pt) node[below]{\tiny{0}};
         \fill (-0.75,0) circle (1.5pt);
      \draw(-1.5, 0) node[below]{\tiny{$-\infty$}};
\end{tikzpicture}
  &&
  {\begin{tikzpicture}
\draw[thick] (-1,0)--(0,0);
   \draw [thick,domain=-30:30] plot ({cos(\x)-1.5}, {sin(\x)});
   \draw [thick,domain=150:210] plot ({cos(\x)+1.5}, {sin(\x)});
   \fill ({cos(30)-1.5},{sin(30)}) circle (1.5pt);
    \fill ({cos(30)-1.5},{-sin(30)}) circle (1.5pt) ; 
   \fill (0,0) circle (1.5pt) node[below]{\tiny{0}};
      \fill (0.5,0) circle (1.5pt) node[right]{\tiny{$3e_1$}};
         \fill ({cos(150)+1.5},{sin(150)}) circle (1.5pt) ; 
    \fill ({cos(210)+1.5},{-sin(150)}) circle (1.5pt); 
      \fill(-0.5,0) circle (1.5pt);
    \draw(-1, 0.1) node[below]{\tiny{$-\infty$}};
\end{tikzpicture}}
&&
{\begin{tikzpicture}
\draw[thick] (-0.25,0)--(0.75,0);
\draw[thick]  (1.1,0)--(1.6,0);
   \draw [thick,domain=-30:30] plot ({cos(\x)-0.65}, {sin(\x)});
   \draw [thick,domain=150:210] plot ({cos(\x)+2.35}, {sin(\x)});
   \fill ({cos(30)-0.65},{sin(30)}) circle (1.5pt);
    \fill ({cos(30)-0.65},{-sin(30)}) circle (1.5pt); 
   \fill(0.75,0) circle (1.5pt) node[below]{\tiny{0}};
      \fill (0.35,0) circle (1.5pt);
         \fill ({cos(150)+2.35},{sin(150)}) circle (1.5pt);
    \fill ({cos(210)+2.35},{-sin(150)}) circle (1.5pt) ; 
    \fill (1.1,0) circle (1.5pt);
      \fill (1.6,0) circle (1.5pt);
      \fill (1.35,0) circle (1.5pt);
      \draw(-0.25, 0.1) node[below]{\tiny{$-\infty$}};
\end{tikzpicture}}&&\\
\scriptsize{$\beta-\delta< b<\beta$} &&
\scriptsize{$b=\beta$}&&\scriptsize{$b>\beta$}&&\\
\hline  \\
&&Figure 3 &&&&
\end{tabular}}
 \end{center}
 \medskip

 \begin{enumerate}
 \item If $0<b<k_1$, then $\vartheta_-<3e_1<\vartheta_+<0$ and
     $$\sigma(L_b^3)=(-\infty, \vartheta_-]\cup[\vartheta_+, 0]\cup \sigma_1\cup \sigma_2$$
     with $\sigma_1\cap (0,+\infty)=\{\text{one point}\}$ and  $\sigma_2\cap (\vartheta_-, \vartheta_+)=\{\text{one point}\}$.
      \item If $b=k_1$, then $\vartheta_-<3e_1<\vartheta_+<0$ and
     $$\sigma(L_b^3)=(-\infty, \vartheta_-]\cup[\vartheta_+, 0]\cup \sigma_1\cup \sigma_2$$
     with $\sigma_1\cap (0,+\infty)=\{\text{one point}\}$ and  $\sigma_2\cap (-\infty,\vartheta_-]=\{\vartheta_-\}$.
     \item If $k_1<b<\widehat{\beta}$, then $\vartheta_-<3e_1<\vartheta_+<0$ and
     $$\sigma(L_b^3)=(-\infty, \vartheta_-]\cup[\vartheta_+, 0]\cup \sigma_1\cup \sigma_2$$
     with $\sigma_1\cap (0,+\infty)=\{\text{one point}\}$ and  $\sigma_2\cap (-\infty, \vartheta_-)=\{\l_-\}$. 
        \item If $b=\widehat{\beta}$, then  $\vartheta_-=\vartheta_+=3e_1<0$ and
     $$\sigma(L_{\widehat{\beta}}^3)=(-\infty, 0]\cup \sigma_1\cup \sigma_2$$
     with $\sigma_1\cap (0,+\infty)=\{\text{one point}\}$ and  $\sigma_2\cap (-\infty, 0)=\{\l_-\}$. 
      \item If $\widehat{\beta}<b<\alpha$, then $f(\lambda)$ has three real roots $\l_-<\l_-'<0<\l_+$ and
     $$\sigma(L_b^3)=(-\infty, 0]\cup \sigma_1\cup \sigma_2\cup \sigma_3$$
     with $\sigma_1\cap (0,+\infty)=\{\text{one point}\}$, $\sigma_2\cap(-\infty, 0)=
     \{\l_-\}$ and $\sigma_3\cap(-\infty,0)=
     \{\l_-'\}$. 
   \item If $b=\alpha$, then
     $$\sigma(L_b^3)=(-\infty, 0]\cup \sigma_1\cup \sigma_4\cup \overline{\sigma_4}$$
     with $\sigma_1\cap (0,+\infty)=\{\text{one point}\}$ and  $\sigma_4\cap(-\infty, 0)=\overline{\sigma_4}\cap(-\infty, 0)=
     \{\l_-\}$. 
      \item If $\alpha<b<\alpha+\varepsilon$, then
     $$\sigma(L_b^3)=(-\infty, 0]\cup \sigma_1\cup \sigma_4\cup \overline{\sigma_4}$$
     with $\sigma_1\cap (0,+\infty)=\{\text{one point}\}$ and  $\sigma_4\cap\R=
     \emptyset$.
      \item If $\beta-\delta\leq b<\beta$, then
     $$\sigma(L_b^3)=(-\infty, 0]\cup \sigma_1\cup \sigma_4\cup \overline{\sigma_4}$$
     with $\sigma_1\cap (-\infty, 0)=\{\lambda_-\}$ and  $\sigma_4\cap\R=
     \emptyset$. 
      \item If $b=\beta$, then $\vartheta_-=\vartheta_+=3e_1>0$ and
     $$\sigma(L_b^3)=(-\infty, 0]\cup \sigma_1\cup \sigma_2$$
     with $\sigma_1\cap \mathbb{R}=\{3e_1\}$ and  $\sigma_2\cap(-\infty, 0)=
     \{\l_-\}$. 
         \item If $b>\beta$, then $0<\vartheta_-<3e_1<\vartheta_+$ and
     $$\sigma(L_b^3)=(-\infty, 0]\cup [\vartheta_-, \vartheta_+]\cup \sigma_1\cup \sigma_2$$
     with $\sigma_1\cap (\vartheta_-, \vartheta_+)=\{\l_+\}$ and  $\sigma_2\cap(-\infty, 0)=
     \{\l_-\}$. 
 \end{enumerate}
 \end{theorem}

\begin{remark}
\begin{enumerate}
    \item The other cases for $b\in [\alpha+\varepsilon, \beta-\delta]$ seem difficult and remain open.
    \item  By comparing Figure 3 with Figure 2, we have a surprising observation. The first 7 graphs in Figure 3 can be obtained from the first 7 graphs in Figure 2 by adding the spectral arc $\sigma_1$ respectively, and the last 3 graphs in Figure 3 can be obtained from the last 3 graphs in Figure 2 by adding the spectral arc $\sigma_2$ respectively.
\end{enumerate}
\end{remark}


\section{Spectral theory of Lam\'e operators}
\label{sec-pre}

In this section, we briefly review some preliminary results about the spectral theory of Lam\'e operators $L_\tau^g$ that are needed in later sections.

Recall the classical Lam\'e equation
\begin{equation*}
   \mathcal{L}(g,\tau,\lambda): \quad y''(z)-g(g+1)\wp(z;\tau) y(z)=\lambda y(z),  \qquad z\in \mathbb{C}.
\end{equation*}
 Let $Y(z; \tau, \lambda)=\left(y_1(z;\tau, \lambda),y_2(z;\tau, \lambda)\right)$ be a fundamental matrix of $\mathcal{L}(g, \tau,\lambda)$ near a fixed base point $p_0\in T_\tau^*:=T_\tau\setminus \{0\}$. In general, $Y(z; \tau, \lambda)$ is multiple-valued with respect to $z$ and might have branch points at $0$. Note that $g\in \Z$, it is well-known (cf. \cite{GW1}) that 
 any solution of $\mathcal{L}(g, \tau, \lambda)$ is meromorphic in $\mathbb{C}$. Hence, $Y(z; \tau, \lambda)$ has no branch points at $0$.
For any loop $\gamma\in \pi_1(T_\tau, p_0)$, denote by $\gamma^*Y(z; \tau, \lambda)$ the analytic continuation of $Y(z; \tau, \lambda)$ along the loop $\gamma$, then there exists a matrix $M(\gamma;\tau, \lambda)\in \mathrm{SL}(2, \mathbb{Z})$ such that $\gamma^*Y(z; \tau, \lambda)=Y(z;\tau, \lambda)M(\gamma;\tau, \lambda)$.  This induces a group homomorphism
$$\rho(\tau, \lambda): \pi_1(T_\tau, p_0)\to \mathrm{SL}(2, \mathbb{Z})$$
which is called the monodromy representation of $\mathcal{L}(g, \tau, \lambda)$. By the deck transformations, it is clear to see that the monodromy group is generated by two matrices $M_1(\tau, \lambda), \, M_2(\tau, \lambda) \in \mathrm{SL}(2, \mathbb{Z})$ satisfying
$$Y(z+\omega_i; \tau, \lambda)=Y(z;\tau, \lambda)M_i(\tau, \lambda), \qquad i=1, 2,$$
and $M_1M_2=M_2M_1$ because $\pi_1(T_\tau, p_0)\cong \mathbb{Z}^2$ is abelian.
So $\rho(\tau, \lambda)$ is reducible and then $M_1(\tau, \lambda), \,M_2(\tau, \lambda)$ can be normalized to satisfy one of the following  two cases (see \cite[Section 2]{CKL1}):

\begin{enumerate}
    \item If $\rho(\tau, \lambda)$ is completely reducible, then
    \begin{equation*}
        M_1(\tau, \lambda)=\left(\begin{array}{cc}e^{-2\pi i s}& 0\\0&e^{2\pi i s}\end{array}\right),\,
         M_2(\tau, \lambda)=\left(\begin{array}{cc}e^{2\pi i r}& 0\\0&e^{-2\pi i r}\end{array}\right),\, (r,s)\in \mathbb{C}^2\setminus \tfrac{1}{2}\mathbb{Z}^2;
    \end{equation*}
 \item If  $\rho(\tau, \lambda)$ is not completely reducible, then
  \begin{equation*}
        M_1(\tau, \lambda)=\epsilon_1\left(\begin{array}{cc}1& 0\\1&1\end{array}\right),\quad
         M_2(\tau, \lambda)=\epsilon_2\left(\begin{array}{cc}1& 0\\ \mathcal{C}&1\end{array}\right),\quad \epsilon_j=\pm 1, \quad \mathcal{C}\in  \mathbb{C}\cup \{\infty\};
    \end{equation*}
    When $\mathcal{C}=\infty$, the monodromy matrices are understood as
     \begin{equation*}
        M_1(\tau, \lambda)=\epsilon_1\left(\begin{array}{cc}1& 0\\0&1\end{array}\right),\quad
         M_2(\tau, \lambda)=\epsilon_2\left(\begin{array}{cc}1& 0\\1&1\end{array}\right).
    \end{equation*}
\end{enumerate}

In particular, $y_1(z; \tau, \lambda)$ can be chosen as a common eigenfunction of $M_1(\tau, \lambda)$ and $M_2(\tau, \lambda)$.
By a direct computation, the local exponent of $\mathcal{L}(g, \tau, \lambda)$ at $0$ is either $-g$ or $g+1$, then a classic theorem says that up to a constant, $y_1(z;\tau, \lambda)$ can be written as
\begin{equation*}
        y_{\mathbf{a}}(z;c)=e^{cz}\frac{\prod_{i=1}^g\sigma(z-a_i)}{\sigma(z)^g},
    \end{equation*}
with $c\in \C$ and $\mathbf{a}=\{a_1, \cdots, a_g\}\in \mathrm{Sym}^g (\mathbb{C}\setminus \Lambda_\tau)$. 

\begin{remark}
If $y_\a(z;c)$ is a solution of $\mathcal{L}(g,\tau, \lambda)$, then $c$ and $\l$ are uniquely determined by $\a$. Indeed, if both $y_\a(z;c_1)$ and $y_\a(z;c_2)$ are solutions of $\mathcal{L}(g,\tau, \lambda)$, a direct computation gives us $c_1=c_2:=c_\a$, then $\l$ is uniquely determined by $y_\a(z; c_\a)$ and denoted by $\l_\a$.
\end{remark}

\begin{theorem}\cite[Theorem 6.2]{CLW}\label{thm-1}
  Let $\mathbf{a}=\{a_1, \cdots, a_g\}\in \mathrm{Sym}^g (\mathbb{C}\setminus \Lambda_\tau)$.   Then $y_{\a}(z;c)$
    is a solution of $\mathcal{L}(g, \tau, \lambda)$ if and only if $a_i-a_j\notin\Lambda_\tau$ for all $i\neq j$ and
    \begin{equation}\label{cond-1}
        \sum\limits_{j\not =i}^{g}(\zeta(a_{i}-a_{j})+\zeta(a_{j})-\zeta(a_{i}))=0, \qquad 1\leq i\leq g.
    \end{equation}
    Moreover, if $y_{\a}(z;c)$
    is a solution of $\mathcal{L}(g, \tau, \lambda)$, then
\begin{align}
    c=c_\a&:=\sum\limits_{i=1}^g \zeta(a_i),\\
    \lambda=\lambda_\a&:=(2g-1)\sum\limits_{i=1}^g\wp(a_i).
\end{align}
\end{theorem}

Denote by $[\a]=\{[a_1], \cdots, [a_g]\}$ with $[a_i]:=a_i\, \mathrm{mod} \Lambda_\tau \in T_\tau^*=T_\tau\setminus \{0\}$. From Theorem \ref{thm-1}, if  $[\a']=[\a]$, we have $\l_\a=\l_{\a'}$,  so $\l_{[\a]}:=\l_\a$ is well defined and then     $y_{\a}(z;c_{\a})$, $y_{\a'}(z;c_{\a'})$
  are linearly dependent solutions of $\mathcal{L}(g, \tau, \l_{[\a]})$. Therefore, we often write $\a$ instead of $[\a]$ to simplify notations when no confusion arises.
Define
\begin{equation*}
\begin{aligned}
   Y_\tau^g&:=\left\{[\a]\in \mathrm{Sym}^gT_\tau^*\,|\,y_{\mathbf{a}}(z;c) \text{  is a solution of } \mathcal{L}(g,\tau, \lambda) \text{ for some } c \text{ and } \lambda. \right\}\\
   &=\left\{ [\a] \in \mathrm{Sym}^g T_\tau^*
\left \vert
\begin{array}
[c]{l}%
\text{ }[a_{i}]\neq [a_{j}]\text{ for any }i\not =j,\\
\sum\limits_{j\not =i}^{g}(\zeta(a_{i}-a_{j})+\zeta(a_{j})-\zeta(a_{i}))=0,
\text{
}\forall i.
\end{array}
\right.  \right\}.
\end{aligned}
        \end{equation*}
 Clearly, if $[\a]\in Y_{\tau}^g$, then
 $[-\a]:=\{[-a_{1}],\cdot \cdot \cdot,[-a_{g}]\} \in Y_{\tau}^g$ and $\l_{[-\a]}=\l_{[\a]}$. Note that the Wronskian of  $y_\a(z;c_\a)$ and $y_{-\mathbf{a}}(z; c_{-\mathbf{a}})$  is a constant, then either $[\a]=[-\a]$ or $[\a]\cap [-\a]=\emptyset$. Moreover, $y_{\a}(z)$ and
$y_{-\a}(z)$ are linearly independent  if and only if $[\a]\cap [-\a]=\emptyset$. We have the following branched covering map of degree 2 (see \cite[Theorem 7.4]{CLW}):
\begin{equation}\label{fun-l}
    \begin{aligned}
       \lambda_\cdot: Y_\tau^{g}&\to \C\\
        [\a]&\mapsto \lambda_{[\a]}:=\l_\a.
    \end{aligned}
\end{equation}
We call $[\a]\in Y_{\tau}^g$ a \emph{branch
point} of $Y_{\tau}^g$ if $[\a]=[-\a]$.

\begin{remark}
If $[\a]\in Y_\tau^g$ is not a branch point, then it was proved in \cite[Proposition 5.8.3]{CLW} that
\begin{equation}
\sum_{j\not =i}(\zeta(a_{i}-a_{j})+\zeta(a_{j})-\zeta(a_{i}))=0,\text{
}\forall1\leq i\leq g,\label{ff3-1}%
\end{equation}
 is equivalent to
\begin{equation}\label{eq-nonbranch}
  \sum_{j=1}^{g}\wp^{\prime}(a_{j})\wp(a_{j})^{l}=0,\quad\forall\,0\leq l\leq g-2.
\end{equation}
\end{remark}
By \eqref{eq-nonbranch}, it follows that
for $[\a]\in Y_\tau^g$, we have
\begin{equation*}
g_{[\a]}(z):=\sum_{i=1}^{g}\frac{\wp^{\prime}(a_{i})}{\wp
(z)-\wp(a_{i})}=\frac{C_{[\a]}}{\prod_{i=1}^{n}(\wp(z)-\wp(a_{i}%
))}\label{g=CP-1}%
\end{equation*}
for a constant 
\begin{equation*}
C_{[\a]}=\wp^{\prime}(a_{i})\prod_{j\neq i}(\wp(a_{i})-\wp
(a_{j})),\; \text{ which is independent of }i.\label{Ca-1}%
\end{equation*}
Remark that $C_{[\a]}=0$ if and only if $[\a]$ is a branch point.

\begin{theorem}\label{thm-Q}\cite[Theorem 7.4]{CLW} There is a so-called spectral polynomial $Q_{g}(\l;\tau
)\in \mathbb{Q}[g_{2}(\tau),g_{3}(\tau)][\l]$ of degree $2g+1$ such that if
$[\a]\in Y_\tau^g$, then $C_{[\a]}^{2}=Q_{g}(\l_{[\a]}; \tau)$. Furthermore, $Q_g(\l;\tau)$ is homogeneous of weight $2g+1$ in $\l, g_2, g_3$ when $\l, g_2, g_3$ are given weights $1,2,3$ respectively.
\end{theorem}

This spectral polynomial $Q_{g}(\l; \tau)$ is called the \emph{Lam\'{e} polynomial} in the
literature.
Theorem \ref{thm-Q} implies%
\begin{equation}\label{eq-hy}
    Y_\tau^g\cong \left \{  (\l,C)\text{ }|\text{ }C^{2}=Q_{g}(\l;\tau)\right \} 
\end{equation}
Therefore, $Y_{\tau}^g$ is a
hyperelliptic curve, known as the \textit{Lam\'{e} curve}.
And $[\a]=[-\a]$ is a branch point of $Y_\tau^g$ if and only if $Q_g(\l_{[\a]}; \tau)=0$.

Let $[\a] \in Y_\tau^g$.
The Legendre relation $\tau\eta_1-\eta_2=2\pi i$ implies that there is a unique $(r,s)\in\mathbb{C}^2$ satisfying
\[r+s\tau=a_1+\cdots+a_g \quad \text{and}\quad r\eta_1+s\eta_2=\zeta(a_1)+\cdots+\zeta(a_g),\]
which is equivalent to
\begin{equation}\label{eq-sr}
\begin{aligned}
  \zeta(a_1)+\cdots+\zeta(a_g)-\eta_1(a_1+\cdots+a_g)=-2\pi i s,\\
\tau(\zeta(a_1)+\cdots+\zeta(a_g))-\eta_2(a_1+\cdots+a_g)=2\pi i r.
\end{aligned}
\end{equation}
Furthermore, the transformation law
$\sigma(z+\omega_j)=-e^{(z+\frac{\omega_j}{2})\eta_j}\sigma(z)$ with $j=1,2$ implies
\begin{equation}\label{eq-eigen}
\begin{aligned}y_{\pm \a}(z+1;c_{\pm\a})&=e^{\pm\sum_{j=1}^{g}(\zeta(a_j)-\eta_1a_j)}
y_{\pm \a}(z;c_{\pm\a})=e^{\mp 2\pi is}y_{\pm \a}(z;c_{\pm\a}),\\
y_{\pm \a}(z+\tau;c_{\pm\a})&=e^{\pm\tau\sum_{j=1}^{g}(\zeta(a_j)-\eta_2a_j)}
y_{\pm \a}(z;c_{\pm\a})=e^{\pm 2\pi i r}y_{\pm \a}(z;c_{\pm\a}),
\end{aligned}\end{equation}
namely $y_{\pm \boldsymbol{a}}(z;c_{\pm\a})$ are elliptic of the second kind.
Note that $y_{\pm\a}(z;c_{\pm\a})$ are solutions of $\mathcal{L}(g,{\tau, \l_\a})$, then $y_{\pm \a}(x+z_0;c_{\pm\a})$ are solutions of $L_\tau^gy=\l_{\a}y$.


\textbf{Case 1.} If $[\a]$ is not a branch point, i.e., $[\a]\cap [-\a]=\emptyset$, then $y_\a(x+z_0; c_\a)$ and $y_{-\a}(x+z_0; c_{-\a})$ are linearly independent solutions of $L_\tau^gy=\l_{[\a]}y$ and satisfy
\begin{equation}\label{eq-eigen1}
y_{\pm \a}(x+z_0+1)=e^{\mp 2\pi is}y_{\pm \a}(x+z_0)\quad \text{ with }\,s\in \C\setminus \tfrac{1}{2}\mathbb{Z}. \end{equation}
In this case, the monodromy representation $\rho(\tau, \l_{\a})$ is completely reducible.

\textbf{Case 2.} If $[\a]$ is a branch point, i.e., $[\a]=[-\a]$, then  $y_\a(x+z_0; c_\a)$ and $y_{-\a}(x+z_0; c_{-\a})$ are linearly dependent solutions of $L_\tau^gy=\l_{[\a]}y$. By (\ref{eq-eigen}), we get $2r, 2s\in \Z$.
It was proved in \cite[Theorem 2.6]{CL-PAMS} that there is a solution $y_2(z)$ linearly independent with $y_{\a}(z)$ such that (note $e^{ 2\pi i s}=e^{ -2\pi i s}=\pm1$)
\begin{equation}\label{eq-eigen2}
y_{\a}(z+1)=e^{2\pi i s}y_{\a}(z),\text{ \ }y_{2}(z+1)=e^{ 2\pi i s}y_{2}(z)+e^{2\pi i s}\chi_{[\a]}
y_{\a}(z), \end{equation}
where
\begin{equation}\label{eq-chi}
\begin{aligned}
\chi_{[\a]}&=-\sum\limits_{i=1}^n c_{a_i} \left(\wp(a_i)+\eta_1\right), \\
c_{a_i}&= 2 \wp^{\prime \prime}(a_{i})^{-1} \prod_{j \ne i}
(\wp(a_{i}) - \wp(a_{j}))^{-1},\quad \text{if }\; a_i=-a_i,
\\
c_{a_i} &=c_{-a_i} = \wp^{\prime}(a_{i})^{-2} \prod_{j \ne
i,\,i^{*}} (\wp(a_{i}) - \wp(a_{j}))^{-1},\quad \text{if }\; a_i\neq -a_i.
\end{aligned}
\end{equation}
In this case,  $\rho(\tau, \l_{\a})$ is not completely reducible.

\subsection{Spectrum of the Lam\'e operator}

Let $y_{1}(x)$ and $y_{2}(x)$ be any two linearly independent solutions of
\begin{equation}\label{eq-lame-x}
    L_\tau^g y=\l y.
\end{equation}
 Then so do $y_{1}(x+1)$ and $y_{2}(x+1)$ and
hence there is a monodromy matrix $M(\l)\in SL(2,\mathbb{C})$ such that
\begin{equation*}
    (y_{1}(x+1),y_{2}(x+1))=(y_{1}(x), y_{2}(x))M(\l).
\end{equation*}
Define the \emph{Hill's discriminant} $\Delta (\l)$ by
\begin{equation}
\label{trace}
\Delta(\l):=\text{tr}M(\l),
\end{equation}
which is clearly an invariant of (\ref{eq-lame-x}), i.e. does not depend on the choice of linearly
independent solutions.
From (\ref{eq-eigen1}) and (\ref{eq-eigen2}),  the Hill's discriminant is given by
\begin{equation}\label{eq-trace}
\Delta(\l)=e^{-2\pi is}+e^{2\pi is}=e^{\sum_{j=1}^{g}(\zeta(a_j)-\eta_1a_j)}
+{e^{-\sum_{j=1}^{g}(\zeta(a_j)-\eta_1a_j)}}.
\end{equation}
This entire function $\Delta(\l)$ encodes all information of the spectrum
$\sigma(L_\tau^g)$; see e.g. \cite{GW2} and references
therein. Indeed,
 Rofe-Beketov \cite{Rofe-Beketov} proved that  the spectrum $\sigma(L_\tau^g)$  can be described as:
\begin{equation*}
\sigma(L_\tau^g)=\Delta^{-1}([-2,2])=\{\l\in \mathbb{C}\,|\, -2\leq \Delta(\l)\leq2\}.\end{equation*}
This important fact play a key role in this paper.

Clearly $\l$ is a (anti)periodic eigenvalue if and only if $\Delta(\l)=\pm 2$. Define
\[
d(\l):=\text{ord}_{\l}(\Delta(\cdot)^{2}-4).
\]
Then it is well known (cf. \cite[Section 2.3]{Naimark}) that $d(\l)$ equals to \textit{the algebraic multiplicity of (anti)periodic
eigenvalues}. Let $c(\l,x,x_{0})$ and $s(\l,x,x_{0})$ be the principal fundamental
system of solutions of (\ref{eq-lame-x}) satisfying the initial values
\[
c(\l,x_{0},x_{0})=s^{\prime}(\l,x_{0},x_{0})=1,\ c^{\prime}(\l,x_{0}%
,x_{0})=s(\l,x_{0},x_{0})=0.
\]
Then we have
\[
\Delta(\l)=c(\l,x_{0}+1,x_{0})+s^{\prime}(\l,x_{0}+1,x_{0}).
\]
Define
\[
p(\l,x_{0}):=\text{ord}_{\l}s(\cdot,x_{0}+\omega,x_{0}),
\]%
\[
p_{i}(\l):=\min \{p(\l,x_{0}):x_{0}\in \mathbb{R}\}.
\]
It is known  that $p(\l,x_{0})$ is the algebraic
multiplicity of a Dirichlet eigenvalue on the interval $[x_0, x_0+\omega]$,
and $p_{i}(\l)$ denotes the immovable part of $p(\l,x_{0})$ (cf. \cite{GW}). It
was proved in \cite[Theorem 3.2]{GW} that $d(\l)-2p_{i}(\l)\geq0$.

Applying the general result \cite[Theorem 4.1]{GW} to $L_\tau^3$, we obtain
\begin{theorem}\label{thm-spec1}
\cite[Theorem 4.1]{GW}
Let $\tau \in \mathbb{H}$, the spectrum $\sigma(L_\tau^3)$ consists of finitely many
bounded simple analytic arcs $\sigma_{k}$, $1\leq k\leq \tilde{g}$ for some $\tilde
{g}\leq 3$ and one semi-infinite simple analytic arc $\sigma_{\infty}$ which tends to
$-\infty+\langle q\rangle$, with $\langle q\rangle=\int_{x_{0}}^{x_{0}+1}q(x)dx$, i.e.
\[
\sigma(L_\tau^3)=\sigma_{\infty}\bigcup \cup_{k=1}^{\tilde{g}}\sigma_{k}.
\]
Furthermore,
\begin{enumerate}
    \item The finite end points of such arcs are exactly zeros of $Q_3(\cdot;\tau)$ with odd order;
    \item There are exactly $d(\l)$'s semi-arcs of $\sigma(L_\tau^3)$ meeting at each zero $\l$ of $Q_3(\cdot;\tau)$  and \begin{equation}\label{degree-formula}
d(\l)=\text{ord}_{\l}Q_3(\cdot;\tau)+2p_i(\l).
\end{equation}
\end{enumerate}
\end{theorem}

Furthermore, we need the following conclusions about $\sigma(L_\tau^3)$.

\begin{theorem}\label{thm-spec2}
{Let $\tau=\frac{1}{2}+bi$ with $b>0$.
\begin{enumerate}
    \item\cite[Lemma 4.1]{CFL-ADV} The spectrum $\sigma(L_\tau^3)$
    is symmetric with respect to the real line $\R$. In particular, $\sigma_{\infty} \subseteq \R$.
    \item \cite[Theorem 2.2]{GW2} The complement $\mathbb{C}\setminus \sigma(L_\tau^3)$ is path-connected.
\end{enumerate}}\end{theorem}

Note that  the multiplicity of any zero of $Q_3(\l; \tau)$ is at most $2$ by Theorem \ref{thm-spec-poly-zeros}, we have  the following results about the spectrum of $L_\tau^3$.

\begin{theorem}\label{thm-degree}\cite[Theorem 1.3]{CL-JDG}
{Let $\tau\in\mathbb{H}$,
\begin{enumerate}
   \item for any finite endpoint $\l_{\a}$ of $\sigma(L_\tau^3)$,
\[d(\l_{\a})=\operatorname{ord}_{\l_{\a}}(\Delta(\l)^2-4)\geq 3\quad\text{if and only if}\quad \chi_{\a}= 0,\]
where $\chi_{\a}$ is defined in (\ref{eq-chi}).
   \item $\sigma(L_\tau^3)$ has at most one endpoint with $d(\l)\geq 3$.
\end{enumerate}
}\end{theorem}

By  Theorem \ref{thm-spec2} (1) and Theorem \ref{thm-degree} (2),  we obtain the following corollary directly.

\begin{corollary}\label{cor-degree-1}
{Let $\tau=\frac{1}{2}+bi$ with $b>0$,  
any simple zero $\l\in\mathbb{C}\setminus\mathbb{R}$ of $Q_3(\l; \tau)$ satisfies $d(\l)=1$, i.e. such $\l$ can not be an intersection point of spectral arcs. 
}\end{corollary}

\subsection{Mean field equation}
\label{sec-mfe}
The purpose of this section is to study the relation between the spectrum $\sigma(L_\tau^g)$ and the number of even axisymmetric solutions of the mean field equation
\begin{equation}\label{eq-mfe-1}
\triangle u+ e^u=8g\pi \delta_0 \quad\text{on }\; T_{\tau}.
\end{equation}
First of all, we recall the connection between (\ref{eq-mfe-1}) and the generalized Lam\'{e} equation $\mathcal{L}(g,\tau, \l)$ which was studied in \cite{CLW}.
\begin{theorem} \cite[Theorem 3.1]{CKL2} \label{thm-mfe}
The mean field equation \eqref{eq-mfe-1} has an even solution if and only if there exists $\l\in \C$ such that the monodromy of $\mathcal{L}(g, \tau, \l)$ is unitary.
Furthermore, the number of even solutions equals the number of those $\l$'s such that the monodromy of $\mathcal{L}(g,\tau,\l)$ is unitary.
\end{theorem}

For any $\l\in \C$, there exists $[\a]\in Y_\tau^g$ such that $\l=\l_\a$ by the covering map \eqref{fun-l}. 
From \eqref{eq-trace}, we have
$$\Delta(\l_\a)=e^{2\pi is}+e^{-2\pi is},\qquad 2\pi is=\sum_{i=1}^g\left(\eta_1 a_i -\zeta(a_i)\right)\in \mathbb{C},$$
then $\sigma(L_b^g)=\Delta^{-1}([-2,2])=\{\lambda_\a \,|\, s\in\mathbb{R}\} $.
Since the Lam\'e potential is doubly periodic, we can also consider its spectrum along the $\tau$ direction.
From \eqref{eq-sr} and \eqref{eq-eigen}, it is clear to see that the Hill's discriminant along the $\tau$ direction denoted by $\Delta_\tau$ is as follows:
$$\Delta_\tau(\l_\a)=e^{2\pi ir}+e^{-2\pi ir},\qquad 2\pi ir=\sum_{i=1}^g\left(\tau\zeta(a_i)-\eta_2 a_i \right)\in \mathbb{C},$$
then the spectrum along the $\tau$ direction is the following: $$\sigma_\tau(L_b^g)=\Delta_\tau^{-1}([-2,2])=\{\lambda_\a \,|\, r\in\mathbb{R}\} .$$

Let $\tau=\frac{1}{2}+ib$ with $b>0$. We will see the spectrum along the imaginary axis plays the same role as $\sigma_\tau(L_b^g)$. Note that $2\tau-1\in i\mathbb{R}$ and
$$y_{\pm \a}(z+2\tau-1)=e^{\pm 2\pi i(2r+s)}y_{\pm \a}(z),$$
the Hill's discriminant along the imaginary direction is given by
\begin{equation*}
\Delta_i(\l_\a):=e^{2\pi i(2r+s)}+e^{-2\pi i(2r+s)},
\end{equation*}
and then
\begin{equation}\label{eq-sigma2}
\sigma_i(L_b^g):=\Delta_i^{-1}([-2,2])=\{\l_\a\,|\,2r+s\in \mathbb{R}\}.
\end{equation}
Clearly, $\sigma(L_b^g)\cap \sigma_i(L_b^g)=\sigma(L_b^g)\cap \sigma_\tau(L_b^g)$. In particular,
the spectrum  ${\sigma}_i(L_b^g)$ can be obtained from  $\sigma(L_{\frac{1}{4b}}^g)$ by a dilation.
\begin{lemma}\label{lem-dualspec}
Let $\tau=\frac{1}{2}+bi$ with $b>0$. We have $${\sigma}_i(L_b^g)=\tfrac{1}{-4b^2}\sigma(L_{\frac{1}{4b}}^g)$$ and the endpoints of ${\sigma}_i(L_b^g)$ are exactly the
endpoints of $\sigma(L_b^g)$.
\end{lemma}
\begin{proof}
Note that
\[\tfrac{\tau-1}{2\tau-1}=\tfrac{1}{2}+i\tfrac{1}{4b}\quad \text{for }\;\tau=\tfrac12+ib.\]
Since $y_{\pm\a}(z)$ are solutions of $\mathcal{L}(g, \tau, \l_\a)$, then $\widehat{y}_{\pm\a}(z):=y_{\pm\a}((2\tau-1)z)$ satisfies
\begin{align*}
\widehat{y}_{\pm\a}''(z)&=(2\tau-1)^2\left(g(g+1)\wp((2\tau-1)z;\tau)+\l_\a\right)\widehat{y}_{\pm\a}(z)\\
&=\left(g(g+1)\wp(z;\tfrac{\tau-1}{2\tau-1})+(2\tau-1)^2\l_\a\right)\widehat{y}_{\pm\a}(z)\\
&=\left(g(g+1)\wp(z;\tfrac12+i\tfrac1{4b})-4b^2\l_\a\right)\widehat{y}_{\pm\a}(z),
\end{align*}
and
\[\widehat{y}_{\pm\a}(z+ 1)=y_{\pm\a}\left((2\tau-1)z+ (2\tau-1)\right)=e^{\pm 2\pi i (2r+s)}\widehat{y}_{\pm\a}(z).\]
Therefore, we have
\begin{equation}\label{eq-lame101}
\Delta(-4b^2\l_\a;\tfrac{1}{4b})=e^{2\pi i (2r+s)}+e^{-2\pi i (2r+s)}={\Delta}_i(\l_\a;b).
\end{equation}
Consequently, we conclude from (\ref{eq-sigma2}) that
\begin{align}\label{eq-dualspec}
{\sigma}_i(L_b^g)&=\{\l\in \C\,|\, -2\leq \Delta(-4b^2\l;\tfrac{1}{4b})\leq 2 \}\nonumber\\
&=\{\l\in \C\,|\, -4b^2\l\in\sigma(L_{\frac{1}{4b}})\}=\tfrac{1}{-4b^2}\sigma(L_{\frac{1}{4b}}).
\end{align}
Theorem \ref{thm-Q} tells us that $Q_g(\l;b)\in \mathbb{Q}[g_2(b), g_3(b)]$ is homogeneous of weight $2g+1$ in $\l, g_2, g_3$ when $\l, g_2, g_3$ are given weights $1, 2, 3$ respectively.
By the modular properties (cf. \cite{CL-CMP}), 
we have
\begin{align*}
    g_2(\tfrac{1}{4b})={16b^4}g_2(b),\qquad
 g_3(\tfrac{1}{4b})=   {-64b^6}g_3(b),
\end{align*} then
$$Q_g(-4b^2\l; \tfrac{1}{4b^2})=(-4b^2)^{2g+1}Q_g(\l; b).$$
Denote by $$Z\left(Q_g(\l;b)\right):=\left\{\l\in \C\,\left\vert\, Q_b(\l)=0 \right.\right\},$$
combining with \eqref{eq-dualspec}, we have  \begin{equation*}
    Z\left(Q_g(\l; b)\right)=Z\left(Q_g(-4b^2\l;  \tfrac{1}{4b^2})\right)=\tfrac{1}{-4b^2}Z\left(Q_g(\l;  \tfrac{1}{4b^2})\right)\subset \sigma_i(L_b^g),
\end{equation*}
thus the endpoints of $\sigma_i(L_b^g)$ are the same as the endpoints of $\sigma(L_b^g)$.
\end{proof}

By the proof of Lemma \ref{lem-dualspec}, we
 $$Z\left(Q_g(\l; b)\right)\subset \sigma(L_b^g)\cap {\sigma}_i(L_b^g),$$
which contains all finite endpoints of both $\sigma(L_b^g)$ and ${\sigma}_i(L_b^g)$.
Note that $Q_g(\l_\a;b)\neq 0$ implies that $(s,r)\notin \frac12\mathbb{Z}^2$, we have
\begin{equation*}
\begin{aligned}
  \Xi_b:=&\left(\sigma(L_b^g)\cap {\sigma}_i(L^g_b)\right)\setminus Z\left(Q_g(\l;b)\right)\\
  =&\left\{\l_\a\,|\, (s,r)\in \mathbb{R}^2\setminus \tfrac{1}{2}\mathbb{Z}^2\right\}.
\end{aligned}
\end{equation*}
The following theorem establishes the precise connection between even solutions of the mean field equation and the spectrum.
\begin{theorem}\label{thm-mfe-1}
Let $\tau=\frac{1}{2}+bi$ with $b>0$.  The number of even solutions of the mean field equation
\eqref{eq-mfe-1} equals to $\# \Xi_b$. Furthermore, the number of even axisymmetric solutions equals to $\# (\Xi_b\cap\mathbb{R})$.
\end{theorem}

\begin{proof}
Note that the monodromy of $\mathcal{L}(g, \tau, {\l_\a})$ is completely reducible if and only if $Q_g(\l_\a;b)\neq 0$ (cf.  \cite[Theorem 2.7]{CKL1}), then the monodromy of $\mathcal{L}(g, \tau, {\l_\a})$ is unitary if and only if $\l_\a\in\Xi_b$. Together with Theorem \ref{thm-mfe}, we conclude that the number of even solutions of (\ref{eq-mfe-1}) equals to $\# \Xi_b$.

Let $u(z)=u(x,y)$ (here we use complex variable $z=x+iy$) be an even solution of \eqref{eq-mfe-1}, then there exists $\l_\a\in\Xi_b$ (cf. \cite{CLW} ) such that
\[(u_{zz}-\tfrac{1}{2}u_z^2)(z)=-2\left(g(g+1)\wp(z;\tau)+\l_\a\right).\]
Clearly $\tilde{u}(z)=\tilde{u}(x,y):=u(x,-y)=u(\overline{z})$ is also an even solution of (\ref{eq-mfe-1}) and satisfies (note that $u(z)$ is real-valued as a solution of (\ref{eq-mfe-1}))
\begin{align*}
(\tilde{u}_{zz}-\tfrac{1}{2}\tilde{u}_z^2)(z)&
=\overline{(u_{zz}-\tfrac{1}{2}u_z^2)(\overline{z})}\\
&=\overline{-2\left(g(g+1)\wp(\overline{z};\tau)+\l_\a\right)}\\
&=-2\left(g(g+1)\wp(z;\tau)+\overline{\l_\a}\right),
\end{align*}
i.e. $\overline{\l_\a}\in\Xi_b$ if $\l_\a\in\Xi_b$.
From here and the fact stated in Theorem \ref{thm-mfe} that there is a one-to-one correspondence between $\l\in\Xi_b$ and even solutions of (\ref{eq-mfe-1}), we conclude that $\l_\a=\overline{\l_\a}$ if and only if $u(z)=\tilde{u}(z)$, i.e. $u(z)=u(\overline{z})$ is axisymmetric. Therefore, the number of even axisymmetric solutions equals to $\# (\Xi_b\cap\mathbb{R})$.
\end{proof}

By \cite[Lemma 2.3]{F-JST}, the spectrum $\sigma(L_b^g)$ is a horizontal translation of the spectrum of the following Darboux-Treibich-Verdier operator
\begin{align*}
  \widehat{L}_b^g:=\frac{d^2}{dx^2}-g(g+1)\left(\wp(x+z_0;2bi)+\wp(x+\tfrac{1+\tilde{\tau}}{2}+z_0;2bi)\right),
\end{align*}
with  $z_0\in\mathbb{C}$ is chosen such that the potential has no singularities on $\mathbb{R}$. Note that the mean field equation \eqref{eq-mfe-1} has no solutions for $\tau=\tfrac{1}{2}+\tfrac{1}{2}i$ (cf. \cite{CL-AJM}), combining with Theorem \ref{thm-mfe-1} and \cite[Theorem 1.7]{F-JST}, we have the following corollary.

\begin{corollary}\label{cor-even}
Let $\tau=\frac{1}{2}+bi$ with $b>0$. There exists two real numbers $k_1, k_2$ with $0< k_1\leq  k_2<\frac{1}{2}$ such that
\begin{equation}\label{eq-even}
 \# (\Xi_b\cap\mathbb{R})=\left\{   \begin{aligned}
   & 2\qquad \text{ if}\quad b\in (0, k_1)\cup (\tfrac{1}{4k_1}, +\infty)\\
   & 1\qquad \text{ if}\quad b\in [k_1,k_2)\cup (\tfrac{1}{4k_2}, \tfrac{1}{4k_1}]\\
      & 0\qquad \text{ if}\quad b\in  [k_2, \tfrac{1}{4k_2}].
    \end{aligned}\right.
\end{equation}
\end{corollary}

In fact, our calculation for the spectrum supplies another way to compute $k_1$ and $k_2$. In particular, $k_2=b_g$ (See Section \ref{sec-proof}).

\section{Inner intersection points }\label{sec-inner}

In this and  the following sections, we study the spectrum of the $g=3$ Lam\'{e} operator:
\begin{equation*}
L_\tau^3=\frac{d^2}{dx^2}-12\wp(x+z_0;\tau),\quad x\in\mathbb{R}.\end{equation*}
We will prove Theorem \ref{thm-inner} recalled here.
\begin{theorem}[=Theorem \ref{thm-inner}] Let $\tau\in \mathbb{H}$ and $\l_0\in \sigma(L_\tau^3)$ with $Q_3(\l_0; \tau)\neq 0$. Then $\l_0$ is an inner intersection point if and only if  
$\l_0$ satisfies the following cubic equation
\begin{equation*}
    f(\l):=\frac{4}{15} \l^3+\frac{8}{5}\eta_1 \l^2-3g_2 \l+9g_3-6\eta_1 g_2=0.
\end{equation*}
\end{theorem}

\begin{proof}
Let $\l_0 \in \sigma(L_\tau^3)$  with  $Q_3(\l_0; \tau)\neq 0$.
By (\ref{eq-hy}) there is a small neighborhood $U\subset \mathbb{C}$ of $\l_0$ such that $Q_3(\l; \tau)\neq 0$ for $\l\in U$ and $\l\in U$ can be a local coordinate for the hyperelliptic curve $Y_\tau^3$, namely $a_1=a_1(\l), a_2=a_2(\l)$ and $a_3=a_3(\l)$ are holomorphic for $\l\in U$. For all $\l\in U$,  (\ref{fun-l}) tells us that
$$\l=\l_{\a}=5\left(\wp(a_1)+\wp(a_2)+\wp(a_3)\right)$$
and then
$$\wp'(a_1)a_1'(\l)+\wp'(a_2)a_2'(\l)+\wp'(a_3)a_3'(\l)=\frac{1}{5} \quad \text{for}\,\, \l\in U$$ and so \begin{equation}\label{anot0}(a_1'(\l_0), a_2'(\l_0), a_2'(\l_0))\neq (0,0,0).
\end{equation}

Next, note that
$Q_3(\l; \tau)\neq 0$ for $\l\in U$ implies
\begin{equation*}
    \left\{a_1(\l),a_2(\l), a_3(\l)\right\}\cap \left\{-a_1(\l),-a_2(\l), -a_3(\l)\right\}=\emptyset \quad \text{for}\,\l\in U,
\end{equation*}
i.e., $\a(\l)=\left\{a_1(\l),a_2(\l), a_3(\l)\right\}$ is not a branch point of $Y_\tau^g$ for all $\l\in U$.
Hence
\begin{equation}\label{neq-wp}
    \wp(a_i(\l))\neq \wp(a_j(\l)) \qquad \text{for all} \quad \l\in U, \quad 1\leq i<j\leq 3,
\end{equation}
and  (\ref{eq-nonbranch}) holds for $\l\in U$, i.e.,
\begin{equation}\label{eq-nonbranch-3}
\begin{aligned}
\wp^{\prime}(a_{1})+\wp'(a_2)+\wp'(a_3)=0,\\
\wp^{\prime}(a_{1})\wp(a_{1})+\wp^{\prime}(a_{2})\wp(a_{2})+\wp^{\prime}(a_{3})\wp(a_{3})=0.
\end{aligned}
\end{equation}
Taking derivative with respect to $\l$ in (\ref{eq-nonbranch-3}) and evaluating at $\l_0$, we obtain from $(\wp')^2=4\wp^3-g_2\wp-g_3$ and $\wp''=6\wp^2-\frac{g_2}{2}$ that
\begin{align}
\sum\limits_{i=1}^3(6\wp_i^2-\tfrac{g_2}{2})a_i'(\l_0)=0, \label{eq-diff1}\\
\sum\limits_{i=1}^3(10\wp_i^3-\frac{3}{2}g_2\wp_i-g_3)a_i'(\l_0)=0, \label{eq-diff2}
\end{align}
where $\wp_i:=\wp\left(a_i(\l_0)\right)$ for $i=1,2,3$.

By $(\wp')^2=4\wp^3-g_2\wp-g_3$ and
 (\ref{eq-nonbranch-3}), we easily obtain
\begin{equation}\label{three-eq}
    \frac{4\wp_1^3-g_2\wp_1-g_3}{(\wp_2-\wp_3)^2}
    =\frac{4\wp_2^3-g_2\wp_2-g_3}{(\wp_1-\wp_3)^2}
    =\frac{4\wp_3^3-g_2\wp_3-g_3}{(\wp_1-\wp_2)^2}=:\mho,
\end{equation}
which is equivalent to
\begin{equation}\label{eq-system1}
    \left\{\begin{aligned}
    4\wp_1^3-g_2\wp_1-g_3=\mho(\wp_2-\wp_3)^2
    ,\\
    4\wp_2^3-g_2\wp_2-g_3=\mho(\wp_1-\wp_3)^2
    ,\\
    4\wp_3^3-g_2\wp_3-g_3=\mho(\wp_1-\wp_2)^2
    .
    \end{aligned}\right.
\end{equation}
Denote by
\begin{equation*}
       \left\{ \begin{aligned}
            s_1&=\wp_1+\wp_2+\wp_3,\\
            s_2&=\wp_1\wp_2+\wp_1\wp_3+\wp_2\wp_3,\\
            s_3&=\wp_1\wp_2\wp_3,
        \end{aligned}\right.
    \end{equation*}
    we have
    \begin{equation}\label{eq-sym1}
        (x-\wp_1)(x-\wp_2)(x-\wp_3)=x^3-s_1x^2+s_2x-s_3,
    \end{equation}
    then (\ref{eq-system1}) is equivalent to
\begin{numcases}{}
    4s_1\wp_1^2-(4s_2+g_2)\wp_1+4s_3-g_3=\mho(\wp_2-\wp_3)^2
   \label{eq-1},\\
    4s_1\wp_2^2-(4s_2+g_2)\wp_2+4s_3-g_3=\mho(\wp_1-\wp_3)^2
     \label{eq-2},\\
    4s_1\wp_3^2-(4s_2+g_2)\wp_3+4s_3-g_3=\mho(\wp_1-\wp_2)^2
    \label{eq-3}.
    \end{numcases}
 First,  (\ref{eq-1})+(\ref{eq-2})+(\ref{eq-3}) lead to
\begin{equation} \label{eq-4}
    \mho=\frac{4(s_1^3-3s_1s_2+3s_3)-g_2s_1-3g_3}{2s_1^2-6s_2}. 
\end{equation}
Note that $\wp_i\neq \wp_j$ for $i\neq j$, then {(\ref{eq-1})}-{(\ref{eq-2})},
 {(\ref{eq-1})}-{(\ref{eq-3})} and {(\ref{eq-2})}-{(\ref{eq-3})} yield
\begin{numcases}{}
4s_2+g_2-4s_1(\wp_1+\wp_2)=\mho(\wp_1+\wp_2-2\wp_3)
  \label{eq-5},\\
   4s_2+g_2-4s_1(\wp_1+\wp_3)=\mho(\wp_1+\wp_3-2\wp_2)
 \label{eq-6},\\
     4s_2+g_2-4s_1(\wp_2+\wp_3)=\mho(\wp_2+\wp_3-2\wp_1)
    \label{eq-7}.
    \end{numcases}
Next,
  (\ref{eq-5})+(\ref{eq-6})+(\ref{eq-7}) gives us
\begin{equation}\label{eq-8}
     s_2=\frac{2}{3}s_1^2-\frac{1}{4} g_2,
\end{equation}
 and (\ref{eq-5})-(\ref{eq-6}) gives us
\begin{equation}\label{mho}
     \mho=-\frac{4}{3}s_1.
\end{equation}
Combine (\ref{eq-4}), (\ref{eq-8}) and (\ref{mho}), we obtain that
\begin{equation}\label{eq-10}
    s_3=\frac{5}{9}s_1^3-\frac{1}{3}g_2s_1+\frac{1}{4}g_3.
\end{equation}

On the other hand, for any $\l\in U$,  denote by $A(\l):=\sum_{j=1}^{3}(\zeta(a_j)-\eta_1a_j)$. 
Since $\a(\l)\cap -\a(\l)=\emptyset$, by (\ref{eq-trace-nonbranch}), we have that for $\l\in U$,
\begin{equation}\label{eq-diffdelta}
\begin{aligned}
\Delta(\l)&=e^A+e^{-A},\\
       \Delta'(\l)&=(e^A-e^{-A})A',\\
       \Delta''(\l)&=(e^A-e^{-A})A''+\Delta (A')^2,\\
       \Delta'''(\l)&=(e^A-e^{-A})\left(A'''+(A')^3\right)+3\Delta A'A''.
\end{aligned}
\end{equation}


\textbf{Sufficiency.}
Let $\l_0 \in \sigma(L_\tau^3)$ with $Q_3(\l_0; \tau)\neq 0$ be an inner intersection point, then $\l_0$ is met by $2k\geq 4 (k\in \Z)$ semi-arcs of the spectrum.

Consider the local behavior of the spectrum at  $\l_0\in \sigma(L_\tau^3)$:
\begin{equation}\label{eq-2lame-2}\Delta(\l)-\Delta(\l_0)=c(\l-\l_0)^{k}(1+O(|\l-\l_0|)),\quad k\geq 1,\;\, c\neq 0.\end{equation}
If $\Delta(\l_0)\in (-2,2)$, it follows from (\ref{eq-2lame-2}) and $\sigma(L_\tau^3)=\{\l|-2\leq\Delta(\l)\leq 2\}$ that there are precisely $2k$ semi-arcs of $\sigma(L_\tau^3)$ meeting at $\l_0$.
If $\Delta(\l_0)=\pm 2$, then there are precisely $k$ semi-arcs of $\sigma(L_\tau^3)$ meeting at $\l_0$.

If $\Delta(\l_0)=\pm 2$, then our assumption implies $k\geq 4$, i.e. $\Delta'(\l_0)=\Delta''(\l_0)=\Delta'''(\l_0)=0$. Since $\Delta(\l_0)=\pm 2$ implies $e^A=\pm 1$ at $\l_0$,  we obtain $A'(\l_0)=0$.

If $\Delta(\l_0)\in (-2,2)$, then our assumption implies $2k\geq 4$, i.e. $k\geq 2$ and so $\Delta'(\l_0)=0$.
Since $\Delta(\l_0)\neq \pm 2$ implies $e^A\neq \pm 1$ at $\l_0$, again we obtain $A'(\l_0)=0$.

Therefore, we always have $A'(\l_0)=0$, i.e.,
\begin{equation}\label{eq-diff3}
 \left(\wp_1+\eta_1\right)a_1'(\l_0)+\left(\wp_2+\eta_1\right)a_2'(\l_0)+\left(\wp_3+\eta_1\right)a_3'(\l_0)=0.
\end{equation}

Noting from (\ref{anot0}), we conclude from (\ref{eq-diff1}, \ref{eq-diff2}, \ref{eq-diff3}) that the determinant of the matrix
\begin{equation*}\label{eq-matrix}
\begin{aligned}
\Omega:=\begin{pmatrix}
      \wp_1+\eta_1 & \wp_2+\eta_1 &\wp_3+\eta_1 \\
      6\wp_1^2-\tfrac{g_2}{2} & 6\wp_2^2-\tfrac{g_2}{2} & 6\wp_3^2-\tfrac{g_2}{2} \\
      10\wp_1^3-\frac{3}{2}g_2\wp_1-g_3&
       10\wp_2^3-\frac{3}{2}g_2\wp_2-g_3&
        10\wp_3^3-\frac{3}{2}g_2\wp_3-g_3
    \end{pmatrix}
    \end{aligned}\end{equation*}
    vanishes, i.e.,
    \begin{equation*}
        (\wp_2-\wp_1)(\wp_3-\wp_1)(\wp_3-\wp_2)
(60s_3+60\eta_1s_2+5g_2s_1-6g_3-9g_2\eta_1)=0.
    \end{equation*}
By (\ref{neq-wp}), we obtain that
\begin{equation}\label{eq-inner-s}
60s_3+60\eta_1s_2+5g_2s_1-6g_3-9g_2\eta_1=0.
\end{equation}

Plug (\ref{eq-8}), (\ref{eq-10}) and \begin{equation}\label{eq-11}\l_0=5(\wp_1+\wp_2+\wp_3)=5s_1\end{equation} into (\ref{eq-inner-s}), we finally obtain that
    $$\frac{4}{15} \l_0^3+\frac{8}{5}\eta_1 \l_0^2-3g_2 \l_0+9g_3-6\eta_1 g_2=0.$$

{\bf Necessity.} Suppose $\l_0\in\sigma(L_\tau^3)$ 
satisfies  $Q_3(\l_0; \tau)\neq 0$ and
$$\frac{4}{15} \l_0^3+\frac{8}{5}\eta_1 \l_0^2-3g_2 \l_0+9g_3-6\eta_1 g_2=0.$$
 This, together with (\ref{eq-8}), (\ref{eq-10}) and (\ref{eq-11}) implies $\det \Omega=0$. Since $\wp_i\neq \wp_j$ for $i\neq j$, the second row of $\Omega$ is nonzero. Suppose that the last two rows of $\Omega$ are linearly dependent, there is $c\in \C$ such that
 \begin{equation*}
     10\wp_i^3-\frac{3}{2}g_2\wp_i-g_3=c(6\wp_i^2-\tfrac{g_2}{2}), \qquad i=1,2,3,
 \end{equation*}
 then
  \begin{equation}\label{eq-sym2}
     (x-\wp_1)(x-\wp_2)(x-\wp_3)=x^3-\frac{3}{5}c x^2-\frac{3}{20}g_2x+\frac{g_2}{20}c-\frac{1}{10}g_3.
 \end{equation}
 Compare (\ref{eq-sym1}) with (\ref{eq-sym2}), we have
 \begin{equation}
     \begin{aligned}
         s_1&=\frac{3}{5}c,\\
         s_2&=-\frac{3}{20}g_2,\\
         s_3&=\frac{1}{10}g_3-\frac{c}{20}g_2.
     \end{aligned}
 \end{equation}
  Combine these with (\ref{eq-8}) and (\ref{eq-10}), we obtain that
 \begin{equation*}
   27g_3^2=5g_2^3   \qquad  \text{and} \qquad \l_0=5s_1=\frac{9g_3}{2g_2}.
 \end{equation*}
 Therefore, if $\tau \in \mathbb{H}\setminus \left\{\tau \in \mathbb{H} \,|\, 27g_3^2(\tau)=5g_2^3(\tau)\right\}$, the last two rows of $\Omega$ are linearly independent, thus the first row can be linearly spanned by the last two rows. Hence (\ref{eq-diff1}) and (\ref{eq-diff2}) yields (\ref{eq-diff3}). Finally, by the continuity of the left hand side in (\ref{eq-diff3}) with respect to $\tau$, we obtain (\ref{eq-diff3}) holds for all $\tau \in \mathbb{H}$.

If  $\Delta(\l_0)\in (-2,2)$, then we see from (\ref{eq-diff3}) and (\ref{eq-diffdelta}) that $\Delta'(\l_0)=0$, i.e. $k\geq 2$ in (\ref{eq-2lame-2}) and so
there are $2k\geq4$ semi-arcs of $\sigma(L_\tau^3)$ meeting at this $\l_0$. If $\Delta(\l_0)=\pm 2$, then $e^{A}=\pm 1$ at $\l_0$. From here and (\ref{eq-diff3}, \ref{eq-diffdelta}), we see that $\Delta'(\l_0)=\Delta''(\l_0)=\Delta'''(\l_0)=0$. This means $k\geq 4$ in (\ref{eq-2lame-2}) and so
there are $k\geq4$ semi-arcs of $\sigma(L_\tau^3)$ meeting at this $\l_0$. Therefore, $\l_0$ is an inner intersection point.
\end{proof}

\section{Zeros of the spectral polynomial}\label{sec-specpoly}

In this and the following sections, we always assume $\tau=\frac{1}{2}+bi$ with $b>0$. In order to emphasize $\tau=\frac{1}{2}+bi$, we will use $b$ instead of $\tau$ in notations. Sometimes,
we omit the notation $\tau$ freely. 
We will first recall some basic properties for the quantities  $e_1, e_2, e_3, g_2, g_3$ and $ \eta_1$ associated with the Weierstrass elliptic function $\wp(z; \tau)$.



First of all, the second equality in (\ref{eq-wp}) gives us
\begin{align}
&e_1+e_2+e_3=0,\label{sum-e}\\
&g_2=2\left(e_1^2+e_2^2+e_3^2\right),\label{g2-sum-e}\\
&g_3=4e_1e_2e_3\label{g3-times-e}.
\end{align}
Note that $e_1, 
\,\,\overline{e_2}=e_3 \not\in \R$ and (\ref{sum-e}), in what follows, we set
\begin{equation}\label{xy}
e_1=2x, \quad e_2=-x+iy, \quad e_3=-x-iy \quad
\text{with} \,\, x,y\in \R \,\, \text{and}\,\, y\neq 0,\end{equation}
and then
\begin{equation}\label{g2g3xy}
\begin{aligned}
    g_2=&4(3x^2-y^2),\\
    g_3=&8x(x^2+y^2)=4e_1^3-e_1g_2.
    \end{aligned}
\end{equation}
Since 
$e_1\neq e_2\neq e_3\neq e_1$,
it is easy to see that
\begin{equation}\label{neq-g2ek}
g_2-3e_k^2=(e_i-e_j)^2\neq 0,
\qquad \text{for}\quad \{i, j, k\}=\{1,2,3\}.
\end{equation}
In particular,
\begin{equation}\label{neq-g2e1}
g_2-3e_1^2=(e_2-e_3)^2=-4y^2<0, \qquad\text{i.e.,}\quad g_2< 3e_1^2.
\end{equation}
The derivatives of $e_1, g_2$ and $\eta_1$ with respect to $b$ are as follows:
\begin{equation}\label{eq-diff}
\begin{aligned}
e_1'(b)
&=\tfrac{1}{\pi}\left( e_{1}^{2}- \eta
_{1}e_{1}- \tfrac{1}{6}g_{2}\right),
\qquad\text{(see  \cite[(2.15)]{CL-PAMS})},\\
g_2'(b)
&=\tfrac{1}{\pi}(3g_3-2\eta_1g_2)=\tfrac{1}{\pi}(12e_1^3-3e_1g_2-2\eta_1g_2), \quad\text{(see  \cite{CL-E4}),}\\
\eta_1'(b)
&=\tfrac{1}{24\pi}(g_{2}-12\eta_{1}^{2})
, \qquad\text{(see  \cite[(1.5)]{CL-JDG}).}
\end{aligned}
\end{equation}
By (\ref{g2g3xy}), we have
\begin{equation}\label{diff-g3}
    g_3'(b)=\frac{1}{\pi}(\frac{1}{6}g_2^2-3\eta_1g_3)=\frac{1}{\pi}(\frac{1}{6}g_2^2+3e_1\eta_1g_2-12e_1^3).
\end{equation}
Moreover, we have the following conclusions about the derivatives.

\begin{proposition}\label{pro-e}
 \cite[Theorem 1.7]{LW} { We have  $e_1\left(\frac{1}{2}\right)=0$ and
\begin{equation*}
 e_1'(b) >0 \qquad
 \text{for all}\quad b>0.
\end{equation*}}
\end{proposition}

\begin{proposition}\label{pro-g}
 \cite[Corollary 4.4]{CL-E4} {
There exists $b_{g}\approx 0.47\in (\frac{1}{2\sqrt{3}},\frac{1}{2})$ such that
\begin{equation*}
  g_{2}'(b)
  \begin{cases}
  <0 \qquad \text{for}\quad b\in (0, b_{g}),\\
  =0\qquad \text{for}\quad b=b_{g},\\
  >0\qquad \text{for}\quad b\in (b_{g}, \infty).
  \end{cases}
\end{equation*}
And $g_2(b)=0$ if and only if $b\in\left\{\frac{1}{2\sqrt 3}, \frac{\sqrt 3}{2}\right\}$.}
\end{proposition}


\begin{proposition}\label{pro-eta}
 \cite[Theorem 1.7]{LW}
There exists $b_{\eta}\approx 0.24108<\frac{1}{2\sqrt{3}}$ such that
\begin{equation*}
  \eta_1'(b)
  \begin{cases}
  >0 \qquad \text{for}\quad b\in (0, b_{\eta}),\\
  =0\qquad \text{for}\quad b=b_{\eta},\\
  <0\qquad \text{for}\quad b\in (b_{\eta}, +\infty).
  \end{cases}
\end{equation*}\end{proposition}

\begin{remark}
All numerical computations in this paper are based on the $q=e^{2\pi i \tau}=-e^{-2\pi b}$ expansions of $e_1, g_2, \eta_1$ which are recalled here for readers' convenience.
\begin{align*}
e_1(b)&=16\pi^2\left(\frac{1}{24}+
\sum\limits_{k=1}^\infty (-1)^k
\sigma_o^ke^{-2k\pi b}\right), \quad\text{where}\,\,
\sigma_o^k=\sum\limits_{1\leq d|k,\, d \text{ is odd}} d, \\
g_2(b)&=320\pi^4\left(\frac{1}{240}+
\sum\limits_{k=1}^\infty (-1)^k
\sigma_3^ke^{-2k\pi b}\right), \quad\text{where}\,\,
\sigma_3^k=\sum\limits_{1\leq d|k} d^3, \\
\eta_1(b)&=8\pi^2\left(\frac{1}{24}-
\sum\limits_{k=1}^\infty (-1)^k
\sigma_1^ke^{-2k\pi b}\right), \quad\text{where}\,\,
\sigma_1^k=\sum\limits_{1\leq d|k} d.
\end{align*}
By numerical computation,  $\eta_1$ vanishes at $b_0\approx 0.13094$.
\end{remark}

Now we figure out all zeros of the spectral polynomial $Q_3(\l; b)$ and prove Theorem \ref{thm-spec-poly-zeros}.
\begin{proof}[Proof of Theorem \ref{thm-spec-poly-zeros}]
Recall that [cf. (\ref{spectralpolynomial})] 
\begin{equation*}
Q_3(\l):=Q_3(\l;b)
=\l\prod_{k=1}^3\left(\l^2-6e_k\l
+15\left(3e_k^2-g_2\right)\right).
\end{equation*}
First of all, by (\ref{neq-g2ek}), $0$ is not a zero of
$$R_k(\l):=\l^2-6e_k\l
+15\left(3e_k^2-g_2\right),\quad k=1, 2, 3.$$
So $0$ is a simple zero of $Q_3(\l)$. By a direct computation, the resultant of $R_i(\lambda)$ and $R_j(\lambda)$ with $i\neq j$
\begin{equation}
    \mathrm{Res}\left(R_i(\lambda), R_j(\lambda)\right)=-135(e_i-e_j)^4\neq 0,
\end{equation}
thus $R_i(\lambda)$ and $R_j(\lambda)$ cannot have common zeros for any pair $i\neq j$.
Since $\overline{R_2(\l)}=R_3(\l)$, we obtain that the two zeros $\mu, \nu$ of $R_2(\l)$ are not real, i.e.,   $\mu,\nu\in \C\setminus \R$, and then $\overline{\mu}, \overline{\nu}\in \C\setminus \R$ are the zeros of $R_3(\l)$.
We claim that $\mu\neq \nu$, otherwise, the discriminant of $R_2(\l)$ vanishes, i.e.,
$$(6e_2)^2-60(3e_2^2-g_2)=0,$$
which gives us
$$5g_2=12e_2^2\in \R,$$
then (\ref{xy}) yields $x=0$ and thus $e_1=0$, $e_3=-e_2$. Combine with (\ref{g2-sum-e}), we have
$$10(e_2^2+e_2^2)=12e_2^2,$$
which is a contradiction because $e_2\neq 0$!

By Corollary \ref{cor-degree-1}, we have
$
d(\mu)=d(\nu)=d(\overline{\mu})=d(\overline{\nu})=1
$.
Therefore, the multiplicity of any zero of $Q_3(\l; b)$ is at most 2, and a zero of $Q_3(\l; b)$ is of multiplicity $2$ if and only if it is a zero of $R_1(\l)$. Denote by
 $\vartheta_{\pm}:=3e_1\pm\frac{1}{2}
\sqrt{\Delta_{R_1}}$ to be  the zeros of
$$R_1(\l)=\l^2-6e_1\l+15(3e_1^2-g_2),$$where $\Delta_{R_1}$ denotes the discriminant of $R_1(\l)$, i.e.,
$$\Delta_{R_1}:=12(5g_2-12e_1^2)=48(3x^2-5y^2).$$
Note that $g_2<3e_1^2$, then $\vartheta_+$ and $\vartheta_-$ have the same sign when  $\Delta_{R_1}>0$.
Clearly, $\Delta_{R_1}$ is holomorphic with respect to $b$. By Proposition \ref{pro-e} and Proposition \ref{pro-g},  $\Delta_{R_1}$ is strictly increasing for $b\in [b_{g}, \frac{1}{2}]$, i.e.,
\begin{equation*}
\frac{d\Delta_{R_1}}{db}>0, \qquad \text{for}\quad b\in \left[b_{g}, \frac{1}{2}\right].
\end{equation*}
By (\ref{eq-diff}), we have
\begin{equation*}
\begin{aligned}
\frac{d\Delta_{R_1}}{db}&=\frac{12}{\pi}
\left(15g_3-10\eta_1g_2-24e_1^3+24\eta_1e_1^2+4e_1g_2\right)\\
&=\frac{12}{\pi}
\left(36e_1^3-11e_1g_2-10\eta_1g_2+24e_1^2\eta_1\right).
\end{aligned}
\end{equation*}
For $b\in \left(\frac{1}{2}, \frac{\sqrt 3}{2}\right]$, since $e_1>0$ and $g_2\leq 0$, we have
\begin{equation*}
\frac{d\Delta_{R_1}}{db}>0, \qquad \text{for}\quad b\in \left(\frac{1}{2}, \frac{\sqrt{3}}{2}\right].
\end{equation*}For $b\in \left(\frac{\sqrt{3}}{2}, +\infty\right)$, note that
\begin{equation*}
\begin{aligned}
\frac{d\Delta_{R_1}}{db}&=\frac{12}{\pi}
\left(36e_1^3-11e_1g_2-10\eta_1g_2+24e_1^2\eta_1\right),\\
&=\frac{96}{\pi}
\left(3x^3+11xy^2+\left(5y^2-3x^2\right)\eta_1\right),
\end{aligned}
\end{equation*}
and $g_2(b), g_2'(b)>0$ gives us
\begin{equation*}
\eta_1<\frac{3g_3}{2g_2}=\frac{3x(x^2+y^2)}{3x^2-y^2}.
\end{equation*}
If $5y^2-3x^2\geq 0$, it is clear that $\Delta_{R_1}'(b)>0$. If $5y^2-3x^2<0$, we have
\begin{equation*}
\begin{aligned}
\frac{d\Delta_{R_1}}{db}
&=\frac{96}{\pi}
\left(3x^3+11xy^2+\left(5y^2-3x^2\right)\eta_1\right),\\
&>\frac{96}{\pi}
\left(3x^3+11xy^2+\left(5y^2-3x^2\right)\frac{3x(x^2+y^2)}{3x^2-y^2}\right)\\
&=\frac{384(9x^2y^2+y^4)}{\pi(3x^2-y^2)}>0.
\end{aligned}
\end{equation*}
Hence
\begin{equation*}
\frac{d\Delta_{R_1}}{db}>0, \qquad \text{for}\quad b\in \left(\frac{\sqrt{3}}{2}, +\infty\right).
\end{equation*}
Note that
\begin{equation*}
    \lim\limits_{b\to +\infty}\Delta_{R_1}(b)=
    12\left(\frac{20\pi^4}{3}-\frac{48\pi^2}{9}\right)
    =16\pi^4>0,
\end{equation*}
and
\begin{equation*}
   \Delta_{R_1}(b)=12(5g_2-12e_1^2)<0 \qquad \text{for}\quad b\in \left[\frac{1}{2\sqrt{3}}, \frac{\sqrt{3}}{2}\right],
\end{equation*}
then there exists unique $\beta\in  \left[\frac{1}{2\sqrt{3}}, +\infty\right)$ such that
\begin{equation*}
    \Delta_{R_1}(b)
    \begin{cases}
    <0, \qquad \text{for} \quad b\in (\frac{1}{2\sqrt{3}},\beta),\\
    =0,  \qquad \text{for} \quad b=\beta,\\
    >0, \qquad \text{for} \quad b\in (\beta, +\infty).
    \end{cases}
\end{equation*}
In fact, $\beta\approx 1.0979>\frac{\sqrt{3}}{2}$ by numerical computation.
On the other hand,
since $\tau=\frac{1}{2}+ib$, we have
\begin{equation*}
\frac{\tau-1}{2\tau-1}=\frac{1}{2}+i\frac{1}{4b} \qquad
\text{for}\quad \tau=\frac{1}{2}+ib.
\end{equation*}It is well know that \cite[p. 346]{CL-CMP}
\begin{equation}\label{eq-eg-mod}
\begin{aligned}
e_1(\frac{1}{4b} )&=-4b^2e_1(b),\\
g_2(\frac{1}{4b} )&=16b^4g_2(b),
\end{aligned}\end{equation}
then
\begin{equation*}
    \Delta_{R_1}(b)
    =\frac{1}{16b^4}\Delta_{R_1}(\frac{1}{4b}),
\end{equation*}
so there exists unique $\widehat{\beta}=\frac{1}{4\beta}\approx 0.2277\in\left(0,\frac{1}{2\sqrt{3}}\right)$ such that
\begin{equation*}
    \Delta_{R_1}(b)
    \begin{cases}
    <0, \qquad \text{for} \quad b\in (\widehat{\beta}, \frac{1}{2\sqrt{3}}),\\
    =0,  \qquad \text{for} \quad b=\widehat{\beta},\\
    >0, \qquad \text{for} \quad b\in (0, \widehat{\beta}).
    \end{cases}
\end{equation*}
Moreover,
\begin{equation*}
   \Delta_{R_1}'(b)<0 \qquad \text{for} \quad b\in (0,\frac{1}{2\sqrt{3}}).
\end{equation*}
Consequently, 
\begin{equation*}
    \Delta_{R_1}(b)
    \begin{cases}
    >0, \qquad \text{for} \quad b\in (0,\widehat\beta)\cup (\beta, +\infty),\\
    =0,  \qquad \text{for} \quad b\in \{\beta, \widehat\beta\},\\
    <0, \qquad \text{for} \quad b\in (\widehat\beta, \beta).
    \end{cases}
\end{equation*}
If $b\in (\widehat{\beta}, \beta)$, then $\vartheta_-=\overline{\vartheta_+}\not\in \R$.
Corollary \ref{cor-degree-1} gives us
\begin{equation*}
d(\vartheta_+)=d(\vartheta_-)=1.
\end{equation*}

Next, if $\l\in \C$ satisfies $Q_3(\l)=0$, by (\ref{fun-l}) and (\ref{eq-hy}), there exists \emph{unique} (cf. \cite[Proposition 6.4]{CLW}) $\a=\{a_1, a_2, a_3\}\in Y_b^3$ with $\a=-\a$ such that
\begin{equation}\label{la}
    \l=\l_{\a}=5\left(\wp(a_1)+\wp(a_2)+\wp(a_3)\right).
\end{equation}
%
Note that $\a=-\a$ with $\a\in Y_b^3$ yields that $\a=\{\frac{\omega_1}{2},\, \frac{\omega_2}{2},\, \frac{\omega_3}{2}\}$ or $\a\in M$ with
\begin{equation}\label{eq-a}
    M=\left\{\{\frac{\omega_k}{2},\, a,\, -a\}\, \Big|\,  \begin{array}{l}
    a\in T_\tau^*\setminus  \{\frac{\omega_1}{2}, \frac{\omega_2}{2}, \frac{\omega_3}{2}\}, \quad k=1,2,3. \\
         \zeta(a-\frac{\omega_k}{2})+\zeta(2a)-3\zeta(a)+\zeta(\frac{\omega_k}{2})=0,
    \end{array}\right\}.
\end{equation}
If $\a=\{\frac{\omega_1}{2},\, \frac{\omega_2}{2},\, \frac{\omega_3}{2}\}$, then
$$\l_{\{\frac{\omega_1}{2},\, \frac{\omega_2}{2},\, \frac{\omega_3}{2}\}}
=5(e_1+e_2+e_3)=0.$$
By \cite[Example 3.4]{CL-JDG},
\[ \chi_{\{\frac{\omega_1}{2},\, \frac{\omega_2}{2},\, \frac{\omega_3}{2}\}}=\frac{12(3g_3-2\eta_1g_2)}{g_2^3-27g_3^2}=\frac{12\pi g_2'}{g_2^3-27g_3^2},\]
where $g_2^3-27g_3^2=16(e_1-e_2)^2(e_1-e_3)^2(e_2-e_3)^2\neq 0$.
Proposition \ref{pro-g} tells us that $\chi_{\{\frac{\omega_1}{2},\, \frac{\omega_2}{2},\, \frac{\omega_3}{2}\}}=0$ if and only if $b=b_{g}$. By Theorem \ref{thm-degree} and the multiplicity of $0$ is 1, we obtain
that
$d(0)\geq 3$ if and only if $b=b_{g}$,
otherwise, $d(0)=1$.


For $b\in \{\widehat{\beta}, \beta\}$, note that  $3e_1=\vartheta_+=\vartheta_-$  is of multiplicity 2 and then (\ref{degree-formula}) gives
\begin{equation}\label{eq-degree-2}
\begin{aligned}
  d(3e_1)=2+p_i(3e_1)\geq 2.
\end{aligned}
\end{equation}
By a direct computation, we have $\{\frac{1}{2},\, a,\, -a\}\in M$ if and only if $\{\frac{1}{2},\, \bar a,\, -\bar a\}\in M$ and $\{\frac{\omega_2}{2},\, a,\, -a\}\in M$ if and only if $\{\frac{\omega_3}{2},\, \bar a,\, -\bar a\}\in M$. Then the zeros of $R_1(\lambda)$ can be written as $\lambda_\a$ with $\a=\{\frac{1}{2},\, a,\, -a\}\in M$.
In particular, there exists  $a\in T_\tau^*\setminus \{\frac{\omega_1}{2}, \frac{\omega_2}{2}, \frac{\omega_3}{2}\}$ such that
\begin{equation*}
    3e_1=\l_{\{\frac{\omega_1}{2},\, a,\,-a\}}=5\left(e_1+2\wp(a)
    \right),
\end{equation*}
thus \begin{equation}\label{eq-wptheta}
    \wp(a)=-\frac{1}{5}e_1.
\end{equation}
At the same time, the discriminant $\Delta_{R_1}=0$ gives us
\begin{equation}\label{eq-g2e1}
    g_2=\frac{12}{5}e_1^2,
\end{equation}
and then
\begin{equation}\label{eq-g3e1}
    g_3=4e_1^3-e_1g_2=\frac{8}{5}e_1^3.
\end{equation}
Note that $\wp'^2=4\wp^3-g_2\wp-g_3$ and $\wp''=6\wp^2-g_2/2$, by (\ref{eq-chi}), we have
\begin{equation*}
    \begin{aligned}
        -\frac{1}{2}\left(e_1-\wp(a)\right)\chi_{\{\frac{\omega_1}{2},\,a,\,-a\}}
        =\frac{e_1+\eta_1}{\left(6e_1^2-g_2/2\right)\left(e_1-\wp(a)\right)}-
        \frac{\wp(a)+\eta_1}{4\wp(a)^3-g_2\wp(a)-g_3}.
    \end{aligned}
\end{equation*}
Combine with (\ref{eq-wptheta}, \ref{eq-g2e1}, \ref{eq-g3e1}), we obtain that
\begin{equation*}
    -\frac{3}{5}e_1\chi_{\{\frac{\omega_1}{2},\, a,\,-a\}}=\frac{25}{24}\eta_1>0,
\end{equation*}
so $\chi_{\{\frac{\omega_1}{2},\, a,\,-a\}}\neq 0$ which means $d(3e_1)\leq 2$ by Theorem \ref{thm-degree}.
Apply (\ref{eq-degree-2}), we obtain that  \begin{equation*}
    d(3e_1)=2 \qquad \text{for}\quad b\in \{\widehat{\beta}, \beta\}.
\end{equation*}
The proof of Theorem \ref{thm-spec-poly-zeros} is complete.
\end{proof}
Finally,
from Theorem \ref{thm-spec-poly-zeros}, Theorem \ref{thm-spec2} and Proposition \ref{pro-e}, we have the following proposition.
\begin{proposition}\label{prop-neginterval}
For $b\in [\widehat{\beta}, +\infty)$, we have
$(-\infty, 0]\subseteq \sigma(L_b^3).$
\end{proposition}

\section{Proof of the main theorem}\label{sec-proof}
Recall that \begin{equation*}
    f(\l)=\frac{4}{15} \l^3+\frac{8}{5}\eta_1 \l^2-3g_2 \l+9g_3-6\eta_1 g_2.
\end{equation*}
 By Proposition \ref{pro-g}, we have
\begin{equation*}
    f(0)=3\pi g_2'(b)\begin{cases}
    <0 \qquad \text{for} \quad b\in (0, b_g),\\
    =0 \qquad \text{for} \quad b=b_g,\\
    >0 \qquad \text{for} \quad b\in (b_g,+\infty).
    \end{cases}
\end{equation*}
From here and Proposition \ref{prop-neginterval},
we obtain the following lemma.
\begin{lemma}\label{prop-inner}
For any $b\in (b_{g}, +\infty)$, there exists $\l_-<0$ such that $\l_-$ is an inner intersection point of $\sigma(L_b^3)$. Furthermore, the spectrum $\sigma(L_b^3)$ always includes the following piece denoted by $\Sigma$.
\begin{center}
\begin{tikzpicture}
    \draw(-5, 0) node[right]{{{\large \bf$\Sigma:$}}};
\draw[thick] (-2.5,0)--(0,0);
   \draw [thick,domain=-50:50] plot ({cos(\x)-2}, {sin(\x)});
   \fill ({cos(50)-2},{sin(50)}) circle (1.5pt);
    \fill ({cos(50)-2},{-sin(50)}) circle (1.5pt);
   \fill (0,0) circle (1.5pt) node[below]{{0}};
      \fill (-1,0) circle (1.5pt);
      \draw(-0.75, 0) node[below]{{$\l_-$}};
      \draw(-2.45, 0) node[below]{{$-\infty$}};
\end{tikzpicture}
\end{center}
\end{lemma}

Consider the derivative of $f(\l)$ denoted by
$$f'(\l)=\frac{4}{5}\l^2+\frac{16}{5}\eta_1\l-3g_2,$$ we have the following conclusion about the discriminant of $f'(\l)$.
\begin{lemma}
There exist $\gamma_1\in (\frac{1}{2\sqrt{3}}, b_g)$ and $\gamma_2\in (\frac{1}{4b_g}, \frac{\sqrt{3}}{2})$ such that
\begin{equation*}
  \Delta_{f'}=\frac{16}{25}(16\eta_1^2+15g_2)\begin{cases}
    >0 \qquad \text{if} \quad b\in (0, \gamma_1)\cup (\gamma_2, +\infty),\\
    =0 \qquad \text{if} \quad b\in \{\gamma_1, \,\gamma_2\},\\
    <0 \qquad \text{if} \quad b\in (\gamma_1, \gamma_2).
    \end{cases}
\end{equation*}
Consequently, $f$ is strictly increasing over $[\gamma_1, \gamma_2]$  and thus has only one real root for $b\in [\gamma_1, \gamma_2]$.
\end{lemma}

\begin{proof} First of all, it is clear that $\Delta_{f'}(b)>0$ for $b\in (0, \frac{1}{2\sqrt{3}}]\cup [\frac{\sqrt{3}}{2}, +\infty)$ by Proposition \ref{pro-g}. The well-known Fourier expansion of $g_2$ gives us numerically $g_2(\frac12)\approx-76.6\pi^2$. This, together with the modular property of $g_2$ (cf. \eqref{eq-eg-mod}) and $\eta_1(\tfrac{1}{2})=2\pi$,
$\eta_1(\tfrac{1}{2\sqrt{3}})=2\sqrt{3}\pi$ (see e.g. \cite[(4.14)]{CFL-ADV}), implies
\begin{equation*}
\begin{aligned}
 \max_{b\in [b_g,\frac12]}\Delta_{f'}(b)&<\tfrac{16}{25}\left(16\left(\eta_1\left(\tfrac{1}{2\sqrt{3}}\right)\right)^2+15g_2\left(\tfrac12\right)\right)<0,\\
  \max_{b\in [\frac{1}{2},\frac{1}{4b_g}]}\Delta_{f'}(b)&<\tfrac{16}{25}\left(16\left(\eta_1\left(\tfrac{1}{2}\right)\right)^2+15g_2\left(\tfrac{1}{4b_g}\right)\right)\\
  &=\tfrac{16^2}{25}\left(\left(2\pi\right)^2+15b_g^4g_2\left(b_g\right)\right)\\
  &<\tfrac{16^2}{25}\left(4\pi^2+15(\tfrac{1}{2\sqrt 3})^4g_2\left(\tfrac{1}{2}\right)\right)
  <0.  \end{aligned}
\end{equation*}
By Proposition \ref{pro-g} and \ref{pro-eta}, $\Delta_{f'}$ is strictly decreasing over $[\frac{1}{2\sqrt{3}}, b_g]$, hence, there exists unique $\gamma_1\in (\frac{1}{2\sqrt{3}}, b_g)$ such that
\begin{equation}
    \Delta_{f'}(\gamma_1)
    \begin{cases}
    >0, \qquad \text{for}\quad b\in (0, \gamma_1),\\
    =0, \qquad \text{for}\quad b=\gamma_1,\\
    <0, \qquad \text{for}\quad b\in (\gamma_1, \frac{1}{4b_g}].
    \end{cases}
\end{equation}

Since $\Delta_{f'}(\tfrac{1}{4b_g})<0$ and $\Delta_{f'}(\tfrac{\sqrt{3}}{2})>0$, there exist at least one $b\in (\tfrac{1}{4b_g}, \tfrac{\sqrt{3}}{2})$ such that $\Delta_{f'}(b)=0$. On the other hand,
if  $b\in (\tfrac{1}{4b_g}, \tfrac{\sqrt{3}}{2})$ satisfying $\Delta_{f'}(b)=\frac{16}{25}(16\eta_1^2+15g_2)=0$, by \eqref{eq-diff},  we have
\begin{equation}\label{diff-Delta-f-prime}
\begin{aligned}
 \Delta_{f'}'(b)
   & =\frac{16}{25\pi}(45g_3-16\eta_1^3-\frac{86}{3}\eta_1g_2)\\
   & =\frac{16}{25\pi}(45g_3-\frac{41}{3}\eta_1g_2)>0
\end{aligned}
\end{equation}
because $g_3=4e_1|e_2|^2>0$ for $b>\frac12$.  Therefore, there exists unique $b\in (\frac{1}{4b_g}, \tfrac{\sqrt{3}}{2})$ denoted by $\gamma_2$ such that $\Delta_{f'}(\gamma_2)=0$.   The proof is complete.
\end{proof}

Denote by 
$  f'(\l)=\frac{4}{5}\l^2+\frac{16}{5}\eta_1\l-3g_2:=\frac{4}{5}(\l-\l_L)(\l-\l_R)$,
we have
\begin{equation*}
    \begin{cases}
      \l_L+\l_R=-4\eta_1,\\
      \l_L\l_R=-\frac{15}{4}g_2.
    \end{cases}
\end{equation*}
If $\lambda_\cdot=\lambda_L$ or $\lambda_R$, i.e., $f'(\lambda_\cdot)=0$, then
\begin{equation*}
    \begin{aligned}
      f(\l_\cdot)&=\frac{\l_\cdot}{3}\left(3g_2-\frac{16}{5}\eta_1\l_\cdot\right)+\frac{8}{5}\eta_1 \l_\cdot^2-3g_2\l_\cdot-6\eta_1 g_2+9g_3\\
      &=\frac{8}{15}\eta_1\l_\cdot^2-2g_2\l_\cdot-6\eta_1 g_2+9g_3\\
      &=\frac{2}{3}\eta_1\left(3g_2-\frac{16}{5}\eta_1\l_\cdot\right)-2g_2\l_\cdot-6\eta_1 g_2+9g_3\\
      &=-2\left(\frac{16}{15}\eta_1^2+g_2
      \right)\l_\cdot-4\eta_1 g_2+9g_3\\
      &:=-\frac{5}{24}\Delta_{f'}(\l_\cdot-H),
    \end{aligned}
\end{equation*}
where
\begin{equation*}
 H=\frac{24(9g_3-4\eta_1 g_2)}{5\Delta_{f'}}.
\end{equation*}
So
\begin{equation*}
    \begin{aligned}
      f(\l_L)f(\l_R)&=\frac{25}{24^2}\Delta_{f'}^2\left(H-\l_L\right)\left(H-\l_R\right)\\
      &=\frac{25}{24^2}\Delta_{f'}^2\left(H^2+4\eta_1H-\frac{15}{4}g_2\right)\\
      &=\frac{25}{24^2}\Delta_{f'}^2\left(\left(H+2\eta_1\right)^2-\frac{25}{64}\Delta_{f'}
      \right)\\
      &=\frac{25}{24^2}\Delta_{f'}^2\left(\frac{\left(\frac{64}{15}\eta_1^3+9g_3\right)^2}{\frac{25}{24^2}\Delta_{f'}^2}-\frac{25}{64}\Delta_{f'}
      \right)\\
      &=\left(\frac{64}{15}\eta_1^3+9g_3\right)^2-\frac{25^2}{192^2}\Delta_{f'}^3\\
      &=\frac{1}{225}\left(\left(64\eta_1^3+135g_3\right)^2-\left(16\eta_1^2+15g_2\right)^3\right).
    \end{aligned}
\end{equation*}
Clearly, all zeros of $f$ are real if and only if $f(\l_L)f(\l_R)\leq 0$.

Now, let us prove the main Theorem \ref{thm-main}.
\begin{proof}[Proof of Theorem \ref{thm-main}]
First of all, by Theorem \ref{thm-spec-poly-zeros}, the zeros of the spectral polynomial $Q_3(\l; b)$ are
\begin{equation*}
    0, \mu, \overline{\mu}, \nu, \overline{\nu}, \vartheta_+, \vartheta_-,
\end{equation*}
and
  $d(\mu)=d(\overline{\mu})=d(\nu)=d(\overline{\nu})=1$. Hence each of $ \mu, \overline{\mu}, \nu, \overline{\nu}$ is met by a single semi-arc of the spectrum, i.e. they are all  endpoints but not intersection points. In this following argument, we will use Theorem \ref{thm-spec2} frequently that $\mathbb{C}\setminus\sigma(L_b^3)$ is path-connected and $\sigma(L_b^3)$ is symmetric with respect to the real line.

\medskip
\textbf{Step1: The spectrum at $b=\beta$. }
By Theorem \ref{thm-spec-poly-zeros}, we have $\vartheta_-=\vartheta_+=3e_1>0$ and
    $d(0)=1$,
   $d(3e_1)=2$.
From Theorem \ref{thm-spec1}, Theorem \ref{thm-spec2}, Proposition \ref{prop-neginterval} and Lemma \ref{prop-inner}, we obtain the rough graph of $\sigma(L_\beta^3)$
is as follows:

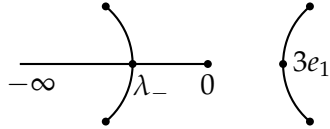
\begin{center}
\begin{tikzpicture}
\draw[thick] (-2.5,0)--(0,0);
   \draw [thick,domain=-50:50] plot ({cos(\x)-2}, {sin(\x)}); 
   \draw [thick,domain=130:230] plot ({cos(\x)+2}, {sin(\x)});
   \fill ({cos(50)-2},{sin(50)}) circle (1.5pt); 
    \fill ({cos(50)-2},{-sin(50)}) circle (1.5pt); 
   \fill (0,0) circle (1.5pt) node[below]{0};
      \fill (1,0) circle (1.5pt) node[right]{$3e_1$};
         \fill ({cos(130)+2},{sin(130)}) circle (1.5pt); 
    \fill ({cos(230)+2},{-sin(130)}) circle (1.5pt); 
      \fill (-1,0) circle (1.5pt);
      \draw(-0.75, 0) node[below]{$\l_-$};
      \draw(-2.35, 0) node[below]{$-\infty$};
\end{tikzpicture}
\end{center}

\medskip

\textbf{Step2: The spectrum for $b>\beta$. } By Theorem \ref{thm-spec-poly-zeros}, we have
 $0<\vartheta_-<3e_1<\vartheta_+$ and
    $d(0)=1$,
which implies again that $0$ is an endpoint but not an intersection point.
By (\ref{degree-formula}), we have
$$d(\vartheta_\pm)=1+2p_i(\vartheta_\pm)$$ is odd, then Theorem \ref{thm-spec2} gives us the closed  interval $[\vartheta_-, \vartheta_+]\subseteq \sigma(L_b^3)$. Hence,
from  Lemma \ref{prop-inner}, Theorem \ref{thm-spec1}, Theorem \ref{thm-spec2} and above analysis, we obtain the graph of $\sigma(L_b^3)$ 
consisting of the following two parts:
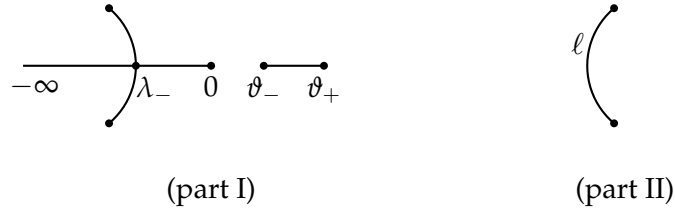
\begin{center}
\begin{tikzpicture}
\draw[thick] (-2.5,0)--(0,0);
\draw[thick]  (0.7,0)--(1.5,0);
   \draw [thick,domain=-50:50] plot ({cos(\x)-2}, {sin(\x)}); 
   \draw [thick,domain=130:230] plot ({cos(\x)+6}, {sin(\x)});
   \fill ({cos(50)-2},{sin(50)}) circle (1.5pt);
    \fill ({cos(50)-2},{-sin(50)}) circle (1.5pt); 
   \fill (0,0) circle (1.5pt) node[below]{0};
      \fill (1.5,0) circle (1.5pt) node[below]{$\vartheta_+$};
         \fill ({cos(130)+6},{sin(130)}) circle (1.5pt); 
    \fill ({cos(230)+6},{-sin(130)}) circle (1.5pt); 
    \fill (0.7,0) circle (1.5pt) node[below]{$\vartheta_-$};
    \fill(5.1,0.3) node[left]{$\ell$};
      \fill (-1,0) circle (1.5pt);
      \draw(-0.75, 0) node[below]{$\l_-$};
         \draw(-2.35, 0) node[below]{$-\infty$};
      \draw(0, -2) node[above]{(part I)};
        \draw(5.5, -2) node[above]{(part II)};
\end{tikzpicture}
\end{center}
and $\ell\cap \R=\{\text{one point}\}$. 
Next, we need the following lemma.
\begin{lemma}
There exists $\l_+\in (\vartheta_-, \vartheta_+)$ such that $f(\l_+)=0$.
\end{lemma}

\begin{proof}
Recall that
\begin{equation*}
    \left\{\begin{aligned}
    R_1(\l)&=\l^2-6e_1\l+15(3e_1^2-g_2)=(\l-\vartheta_-)(\l-\vartheta_+),\\
   f(\l)&=\frac{4}{15}\l^3+\frac{8}{5}\eta_1 \l^2-3g_2\l+9g_3-6\eta_1g_2,
    \end{aligned}\right.
\end{equation*}
and $g_3=4e_1^3-e_1g_2$, we have
\begin{equation}\label{eq-sumdiff1}
    \begin{cases}
    \vartheta_-+\vartheta_+=6e_1,\\
    \vartheta_-\vartheta_+=15(3e_1^2-g_2),
    \end{cases}
\end{equation}
and
\begin{equation*}
   \begin{aligned}
   f(\vartheta_\pm)&=\left(\frac{4}{15}\vartheta_\pm+\frac{8}{5}\eta_1\right)\left(6e_1\vartheta_\pm-15(3e_1^2-g_2)\right)-3g_2\vartheta_\pm+9g_3-6\eta_1g_2\\
   &=\frac{8}{5}e_1\left(6e_1\vartheta_\pm-15(3e_1^2-g_2)\right)+\left(g_2+\frac{48}{5}\eta_1e_1-12e_1^2\right)\vartheta_\pm +9g_3+18\eta_1g_2-72e_1^2\eta_1\\
   &=\left(g_2+\frac{48}{5}\eta_1e_1-\frac{12}{5}e_1^2\right)\vartheta_\pm +3\left(6g_2\eta_1+5g_2e_1-72\eta_1e_1^2-36e_1^3\right)\\
   &:=B(\vartheta_\pm-D),
    \end{aligned}
\end{equation*}
where
\begin{equation}\label{eq-sumdiff2}
   \begin{aligned}
   D&=-3\frac{6g_2\eta_1+5g_2e_1-72\eta_1e_1^2-36e_1^3}{g_2+\frac{48}{5}\eta_1e_1-\frac{12}{5}e_1}\\
  &=-15e_1+\frac{15}{2}\frac{\eta_1(g_2-12e_1^2)}{e_1^2-4e_1\eta_1-\frac{5}{12}g_2}.
    \end{aligned}
\end{equation}
From (\ref{eq-sumdiff1}) and (\ref{eq-sumdiff2}), we have
\begin{equation*}
   \begin{aligned}
  & f(\vartheta_-)f(\vartheta_+)\\
  =&B^2(\vartheta_--D)(\vartheta_+-D)\\
  =&B^2\left(D^2-6e_1D+15(3e_1^2-g_2)\right)\\
  =&B^2\left((D-3e_1)^2-3(5g_2-12e_1^2)\right)\\
  =&B^2\left( \left(-18e_1+\frac{15}{2}\frac{\eta_1(g_2-12e_1^2)}{e_1^2-4e_1\eta_1-\frac{5}{12}g_2}\right)^2-3(5g_2-12e_1^2)\right)\\
 =&\frac{3B^2(5g_2-12e_1^2)}{4(e_1^2-4e_1\eta_1-\frac{5}{12}g_2)^2}\left(3(5g_2-12e_1^2)(e_1+\eta_1)^2-4(e_1^2-4e_1\eta_1-\frac{5}{12}g_2)^2\right)\\
  =&-\frac{\pi B^2\Delta_{R_1}}{288(e_1^2-4e_1\eta_1-\frac{5}{12}g_2)^2}\left(12(e_1+\eta_1)g_2'-5(12e_1^2-g_2)\eta_1'\right)<0 
    \end{aligned}
\end{equation*}
for all $b>\beta$
because $e_1>0, \eta_1>0$, $g_2'>0$, $\eta_1'<0$,  $g_2<3e_1^2$ and $\Delta_{R_1}>0$.
Therefore, there exists $\l_+\in (\vartheta_-, \vartheta_+)$ such that $f(\l_+)=0$.
\end{proof}
From Theorem \ref{thm-inner} and this lemma, $\l_+$ is an inner intersection point and then $$\l_+=\ell\cap\R\in (\vartheta_-, \vartheta_+),$$
so the rough graph of  spectrum $\sigma(L_b^3)$ for $b>\beta$ is as follows. 
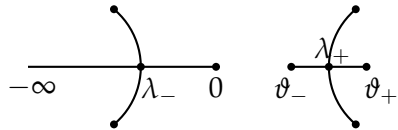
\begin{center}
\begin{tikzpicture}
\draw[thick] (-2.5,0)--(0,0);
\draw[thick]  (1,0)--(2,0);
   \draw [thick,domain=-50:50] plot ({cos(\x)-2}, {sin(\x)}); 
   \draw [thick,domain=130:230] plot ({cos(\x)+2.5}, {sin(\x)});
   \fill ({cos(50)-2},{sin(50)}) circle (1.5pt); 
    \fill ({cos(50)-2},{-sin(50)}) circle (1.5pt); 
   \fill (0,0) circle (1.5pt) node[below]{0};
      \fill (2,0) circle (1.5pt);
      \draw(2.2, 0) node[below]{$\vartheta_+$};
         \fill ({cos(130)+2.5},{sin(130)}) circle (1.5pt); 
    \fill ({cos(230)+2.5},{-sin(130)}) circle (1.5pt); 
    \fill (1,0) circle (1.5pt) node[below]{$\vartheta_-$};
      \fill (-1,0) circle (1.5pt);
      \draw(-0.75, 0) node[below]{$\l_-$};
      \fill (1.5,0) circle (1.5pt);
      \draw(1.55, -0.1) node[above]{$\l_+$};
          \draw(-2.45, 0) node[below]{$-\infty$};
\end{tikzpicture}
\end{center}

\textbf{Step3: The spectrum for $b$ around $\widehat{\beta}$. }

First of all, if $b=\widehat{\beta}$, by Theorem \ref{thm-spec-poly-zeros}, we have $\vartheta_-=\vartheta_+=3e_1<0$ and then $\Delta_{R_1}(\widehat{\beta})=0$, i.e., $5g_2=12e_1^2$, thus $5g_3=8e_1^3$. By direct computations, we have
$f(3e_1)=0$ and $f'(3e_1)=4e_1\eta_1/5<0$ because $e_1(\widehat{\beta})<0$ and $\eta_1(\widehat{\beta})>0$. Note that $f(0)<0$, then $f$ has two negative real zeros $\l_-^1< \l_-^2=3e_1$ and one positive real zero $\l_+$, thus $f(\l_L)f(\l_R)\vert_{b=\widehat{\beta}}<0$.
Note that $f(\l_L)f(\l_R)\vert_{b=b_\eta}>0$ by numerical computation,
we can define
$$\alpha:=\sup{\left\{\widetilde{b}>\widehat{\beta}\quad\vert \quad f(\l_L)f(\l_R)\vert_{b=b'}<0 \, \text{ for all }\, b'\in[\widehat{\beta},\widetilde{b}) \right\}}.$$
Clearly, $\widehat{\beta}<\alpha<b_\eta$ and $f(\l_L)f(\l_R)\vert_{b=\alpha}=0$. In fact, $\alpha\approx 0.23217$ by numerical computation and $f(\l_L)f(\l_R)\vert_{\alpha<b<\alpha+\varepsilon}>0$ with $\varepsilon>0$ sufficient small.

For $b\in [\widehat{\beta},\alpha)$, all zeros of $f$ are real  because $f(\l_L)f(\l_R)\vert_{[\widehat{\beta},\alpha)}<0$.
Since $f(0)=3\pi g_2'(b)<0$ for $b\in  (0, b_g)\supset [\widehat{\beta},\alpha)$ and $f$ has two negative real zeros  at $b=\widehat{\beta}$, by continuity, $f$ has two negative real zeros $\l_-^1, \l_-^2$ for $b\in [\widehat{\beta},\alpha)$ and $\l_-:=\l_-^1= \l_-^2$ is a negative real zero of $f$ with multiplicity $2$ at $b=\alpha$.

For $b\in (\widehat{\beta},b_g)$,
we have $d(0)=1$ by Theorem \ref{thm-spec-poly-zeros}.
Since all other zeros of $Q_3(\l; b)$  are not real,  we obtain that $(-\infty, 0]\subseteq \sigma(L_b^3)$ by Theorem \ref{thm-spec2} and there are exactly $7$ finite endpoints:
$$0, \mu, \overline{\mu}, \nu, \overline{\nu}, \vartheta_+, \vartheta_-$$
 which are  not intersection points by Corollary \ref{cor-degree-1}.

For $b\in (\widehat{\beta},\alpha)\subset[\widehat{\beta},\alpha)\subset(\widehat{\beta},b_g)$,
we have $(-\infty, 0]\subseteq \sigma(L_b^3)$ and $\l_-^1<\l_-^2<0$, thus $\l_-^1$ and $\l_-^2$ are two inner intersection points by Theorem \ref{thm-inner}.
Consequently, the spectrum $\sigma(L_b^3)$ is one of the following graph.
\medskip

    \begin{tabular}{c|c|c}
    \hline\\
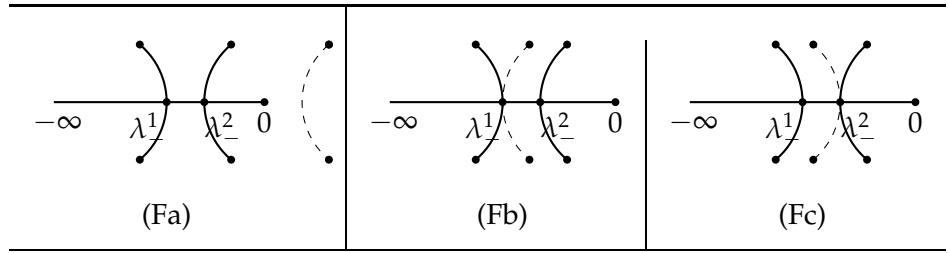
\begin{tikzpicture}
\draw[thick] (-3,0)--(-0.2,0);
   \draw [thick,domain=130:230] plot ({cos(\x)}, {sin(\x)});
      \draw [thick,domain=-50:50] plot ({cos(\x)-2.5}, {sin(\x)});
         \draw [dashed,domain=130:230] plot ({cos(\x)+1.3}, {sin(\x)});
  \fill ({cos(50)-2.5},{sin(50)}) circle (1.5pt); 
    \fill ({cos(50)-2.5},{-sin(50)}) circle (1.5pt); 
    \fill ({cos(130)},{sin(50)}) circle (1.5pt);
      \fill ({cos(130)},{-sin(50)}) circle (1.5pt);
    \fill ({cos(130)+1.3},{-sin(50)}) circle (1.5pt);
    \fill ({cos(130)+1.3},{sin(50)}) circle (1.5pt);
    \fill ({cos(130)+1.3},{sin(50)}) circle (1.5pt);
   \fill (-0.2,0) circle (1.5pt) node[below]{0};
      \fill (-1,0) circle (1.5pt);
      \draw(-0.75, 0) node[below]{$\l_-^2$};
         \fill (-1.5,0) circle (1.5pt);
      \draw(-1.75, 0) node[below]{$\l_-^1$};
          \draw(-2.95, 0) node[below]{$-\infty$};
       \draw(-1.5, -1.2) node[below]{(Fa)};
\end{tikzpicture}
&
\begin{tikzpicture}
\draw[thick] (-3,0)--(0,0);
   \draw [thick,domain=130:230] plot ({cos(\x)}, {sin(\x)});
      \draw [thick,domain=-50:50] plot ({cos(\x)-2.5}, {sin(\x)});
         \draw [dashed,domain=130:230] plot ({cos(\x)-0.5}, {sin(\x)});
\fill ({cos(50)-2.5},{sin(50)}) circle (1.5pt);
    \fill ({cos(50)-2.5},{-sin(50)}) circle (1.5pt);
    \fill ({cos(130)-0.5},{sin(50)}) circle (1.5pt);
    \fill ({cos(130)-0.5},{-sin(50)}) circle (1.5pt);
    \fill ({cos(130)},{sin(50)}) circle (1.5pt);
    \fill ({cos(130)},{-sin(50)}) circle (1.5pt);
   \fill (0,0) circle (1.5pt) node[below]{0};
      \fill (-1,0) circle (1.5pt);
      \draw(-0.75, 0) node[below]{$\l_-^2$};
         \fill (-1.5,0) circle (1.5pt);
      \draw(-1.75, 0) node[below]{$\l_-^1$};
          \draw(-2.95, 0) node[below]{$-\infty$};
       \draw(-1.5, -1.2) node[below]{(Fb)};
\end{tikzpicture}
&
\begin{tikzpicture}
\draw[thick] (-3,0)--(0,0);
   \draw [thick,domain=130:230] plot ({cos(\x)}, {sin(\x)});
      \draw [thick,domain=-50:50] plot ({cos(\x)-2.5}, {sin(\x)});
       \draw [dashed,domain=-50:50] plot ({cos(\x)-2}, {sin(\x)});
  \fill ({cos(50)-2.5},{sin(50)}) circle (1.5pt);
    \fill ({cos(50)-2.5},{-sin(50)}) circle (1.5pt);
    \fill ({cos(130)},{sin(50)}) circle (1.5pt);
    \fill ({cos(130)},{-sin(50)}) circle (1.5pt);
    \fill ({cos(50)-2},{sin(50)}) circle (1.5pt);
    \fill ({cos(50)-2},{-sin(50)}) circle (1.5pt);
   \fill (0,0) circle (1.5pt) node[below]{0};
      \fill (-1,0) circle (1.5pt);
      \draw(-0.75, 0) node[below]{$\l_-^2$};
         \fill (-1.5,0) circle (1.5pt);
      \draw(-1.75, 0) node[below]{$\l_-^1$};
          \draw(-2.95, 0) node[below]{$-\infty$};
       \draw(-1.5, -1.2) node[below]{(Fc)};
\end{tikzpicture}
\\
\hline
    \end{tabular}

\medskip

When $b$ goes to $\alpha$ from the left hand side, we always have $(-\infty, 0]\subseteq \sigma(L_b^3)$ and $d(0)=1$, i.e., $0$ is always an endpoint of the arc $(-\infty, 0]$ and cannot be an intersection point.  At the same time,
$\l_1, \l_2$ will come together to be $\l_-$. Also, the graphs are symmetric with respect to $\R$ by Theorem \ref{thm-spec2}, and there is no non-real inner intersection point because all zeros of $f$ are real.
Hence, (Fa) will goes to (Ga), (Fb) and (Fc) will goes to either (Gb) or (Gc).

\medskip

    \begin{tabular}{c|c|c}
    \hline\\
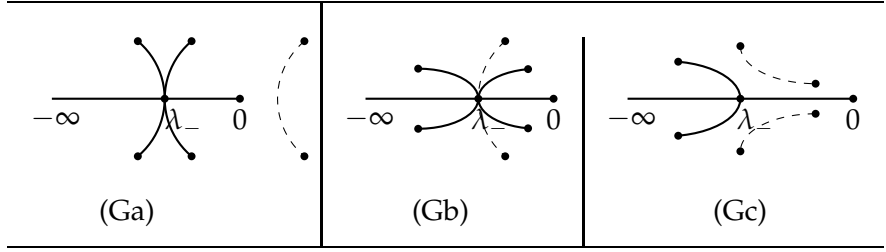
\begin{tikzpicture}
\draw[thick] (-2.5,0)--(0,0);
   \draw [thick,domain=130:230] plot ({cos(\x)}, {sin(\x)});
      \draw [thick,domain=-50:50] plot ({cos(\x)-2}, {sin(\x)});
         \draw [dashed,domain=130:230] plot ({cos(\x)+1.5}, {sin(\x)});
  \fill ({cos(50)-2},{sin(50)}) circle (1.5pt);
    \fill ({cos(50)-2},{-sin(50)}) circle (1.5pt);
    \fill ({cos(130)+1.5},{sin(130)}) circle (1.5pt);
     \fill ({cos(130)+1.5},{-sin(130)}) circle (1.5pt);
    \fill ({cos(130)},{-sin(130)}) circle (1.5pt);
    \fill ({cos(130)},{sin(130)}) circle (1.5pt);
   \fill (0,0) circle (1.5pt) node[below]{0};
      \fill (-1,0) circle (1.5pt);
     \draw(-0.75, 0) node[below]{$\l_-$};
         \draw(-2.45, 0) node[below]{$-\infty$};
       \draw(-1.5, -1.2) node[below]{(Ga)};
\end{tikzpicture}
&
\begin{tikzpicture}
\draw[thick] (-2.5,0)--(0,0);
   \draw [thick,domain=100:260] plot ({0.8*cos(\x)-0.2}, {0.4*sin(\x)});
      \draw [dashed,domain=130:230] plot ({cos(\x)}, {sin(\x)});
         \draw [thick,domain=-90:90] plot ({0.8*cos(\x)-1.8}, {0.4*sin(\x)});
  \fill ({0.8*cos(100)-0.2},{0.4*sin(100)}) circle (1.5pt);
    \fill ({0.8*cos(260)-0.2},{-0.4*sin(100)}) circle (1.5pt);
    \fill ({0.8*cos(90)-1.8},{0.4*sin(90)}) circle (1.5pt);
     \fill ({0.8*cos(90)-1.8},{-0.4*sin(90)}) circle (1.5pt);
    \fill ({cos(130)},{-sin(130)}) circle (1.5pt);
    \fill ({cos(130)},{sin(130)}) circle (1.5pt);
   \fill (0,0) circle (1.5pt) node[below]{0};
      \fill (-1,0) circle (1.5pt);
      \draw(-0.85, 0) node[below]{$\l_-$};
          \draw(-2.45, 0) node[below]{$-\infty$};
       \draw(-1.5, -1.2) node[below]{(Gb)};
\end{tikzpicture}
&
\begin{tikzpicture}
\draw[thick] (-3,0)--(0,0);
       \draw [thick,domain=-80:80] plot ({cos(\x)-2.5}, {0.5*sin(\x)});
   \fill (0,0) circle (1.5pt) node[below]{0};
         \fill (-1.5,0) circle (1.5pt);
      \draw(-1.3, 0) node[below]{$\l_-$};
      \draw [dashed,domain=-180:-90] plot ({cos(\x)-0.5}, {0.5*sin(\x)+0.7});
    \draw [dashed,domain=90:180] plot ({cos(\x)-0.5}, {0.5*sin(\x)-0.7});
    \fill ({cos(80)-2.5},{0.5*sin(80)}) circle (1.5pt);
    \fill ({cos(80)-2.5},{-0.5*sin(80)}) circle (1.5pt);
    \fill ({cos(90)-0.5},{-0.5*sin(90)+0.7}) circle (1.5pt);
     \fill ({cos(180)-0.5},{-0.5*sin(180)+0.7}) circle (1.5pt);
    \fill ({cos(90)-0.5},{0.5*sin(90)-0.7}) circle (1.5pt);
    \fill ({cos(180)-0.5},{0.5*sin(180)-0.7}) circle (1.5pt);
        \draw(-2.95, 0) node[below]{$-\infty$};
       \draw(-1.5, -1.2) node[below]{(Gc)};
\end{tikzpicture}
\\
\hline
    \end{tabular}
    \medskip

Since $f(\l_L)f(\l_R)\vert_{\alpha<b<\alpha+\varepsilon}>0$ and $f(0)<0$, $f$ has no negative real zero for $b\in(\alpha, \alpha+\varepsilon)$, i.e., there is no negative inner intersection point for $b\in(\alpha, \alpha+\varepsilon)$. By Theorem \ref{thm-spec1} and Theorem \ref{thm-spec2},  $\l_-$ can not disappear in the deformation of (Gb) and (Gc) when $b$ moves to the right hand side of $\alpha$, this is a contradiction!
Therefore, the rough graph of $\sigma(L_b^3)$ at $b=\alpha$ is (Ga), and then the rough graph of $\sigma(L_b^3)$ for $b\in (\widehat{\beta}, \alpha)$ is (Fa). Consequently, the rough graph of $\sigma(L_b^3)$ for $b\in (\alpha, \alpha+\varepsilon)$ with $\varepsilon>0$ sufficient small is the following:
    
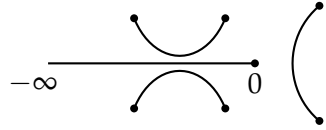
\begin{center}
        \begin{tikzpicture}
\draw[thick] (-2.75,0)--(0,0);
    \draw [thick,domain=-150:-30] plot ({0.7*cos(\x)-1}, {sin(\x)+1.1});
    \draw [thick,domain=30:150] plot ({0.7*cos(\x)-1}, {sin(\x)-1.1});
         \draw [thick,domain=130:230] plot ({cos(\x)+1.5}, {sin(\x)});
  \fill ({0.7*cos(30)-1},{sin(30)-1.1}) circle (1.5pt);
    \fill ({0.7*cos(150)-1},{sin(150)-1.1}) circle (1.5pt);
    \fill ({cos(130)+1.5},{sin(130)}) circle (1.5pt);
     \fill ({cos(130)+1.5},{-sin(130)}) circle (1.5pt);
    \fill ({0.7*cos(30)-1},{-sin(30)+1.1}) circle (1.5pt);
    \fill ({0.7*cos(150)-1},{-sin(150)+1.1}) circle (1.5pt);
   \fill (0,0) circle (1.5pt) node[below]{0};
     \draw(-2.95, 0) node[below]{$-\infty$};
\end{tikzpicture}
    \end{center}

Now let us go back to $b=\widehat{\beta}$, recall that $f$ has two negative real zeros $\l_-^1<\l_-^2=3e_1$ and
 $d(3e_1)=2$ by Theorem \ref{thm-spec-poly-zeros}, we obtain the rough graph of $\sigma(L_b^3)$ at $b=\widehat{\beta}$ from (Fa) as follows:
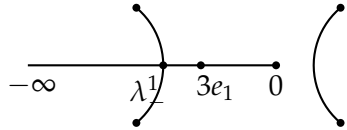
\begin{center}
\begin{tikzpicture}
\draw[thick] (-3.3,0)--(0,0);
      \draw [thick,domain=-50:50] plot ({cos(\x)-2.5}, {sin(\x)});
         \draw [thick,domain=130:230] plot ({cos(\x)+1.5}, {sin(\x)});
  \fill ({cos(50)-2.5},{sin(50)}) circle (1.5pt);
    \fill ({cos(50)-2.5},{-sin(50)}) circle (1.5pt);
    \fill ({cos(130)+1.5},{sin(130)}) circle (1.5pt);
     \fill ({cos(130)+1.5},{-sin(130)}) circle (1.5pt);
   \fill (0,0) circle (1.5pt) node[below]{0};
      \fill (-1,0) circle (1.5pt);
      \draw(-0.8, 0) node[below]{$3e_1$};
         \fill (-1.5,0) circle (1.5pt);
      \draw(-1.7, 0) node[below]{$\l_-^1$};
      \draw(-3.25, 0) node[below]{$-\infty$};
\end{tikzpicture}
\end{center}
Note that $\sigma_i(L^3_{\widehat\beta})=-4\widehat\beta^2\sigma(L^3_\beta)$, we draw the rough graphs of $\sigma_i(L^3_{\widehat\beta})$ using blue dashed lines and $\sigma(L^3_{\widehat\beta})$ using dark solid lines in the following:
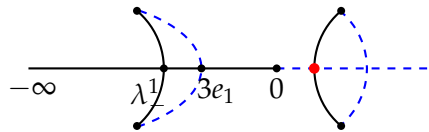
\begin{center}
\begin{tikzpicture}
\draw[thick] (-3.3,0)--(0,0);
\draw[thick,dashed,blue] (2,0)--(0,0);
   \draw [thick,dashed,blue,domain=-55:55] plot ({2*cos(\x)-3}, {0.95*sin(\x)});
      \draw [thick,domain=-50:50] plot ({cos(\x)-2.5}, {sin(\x)});
         \draw [thick,domain=130:230] plot ({cos(\x)+1.5}, {sin(\x)});
                 \draw [thick,dashed,blue,domain=-52:52] plot ({cos(\x)+0.2}, {sin(\x)});
  \fill ({cos(50)-2.5},{sin(50)}) circle (1.5pt);
    \fill ({cos(50)-2.5},{-sin(50)}) circle (1.5pt);
    \fill ({cos(130)+1.5},{sin(130)}) circle (1.5pt);
     \fill ({cos(130)+1.5},{-sin(130)}) circle (1.5pt);
   \fill (0,0) circle (1.5pt) node[below]{0};
      \fill (-1,0) circle (1.5pt);
      \draw(-0.8, 0) node[below]{$3e_1$};
         \fill (-1.5,0) circle (1.5pt);
            \fill[red] (0.5,0) circle (2pt);
      \draw(-1.7, 0) node[below]{$\l_-^1$};
      \draw(-3.25, 0) node[below]{$-\infty$};
\end{tikzpicture}
\end{center}
From this graph, it is clear to see that the red point is the only element of $\Xi_{\widehat{\beta}}\cup \mathbb{R}$, so $\#(\Xi_{\widehat{\beta}}\cup \mathbb{R})=1$.

Furthermore, if $b\in (\widehat{\beta}-\delta, \widehat{\beta})$ with $\delta>0$ sufficient small, then $\vartheta_-<3e_1<\vartheta_+<0$ by Theorem \ref{thm-spec-poly-zeros} and $f$ still has two negative real zeros $\l_-^1<\l_-^2$ with $\l_-^1$ still lying on $\sigma(L_b^3)$. Hence, the rough graph of $\sigma(L_b^3)$ is the following:
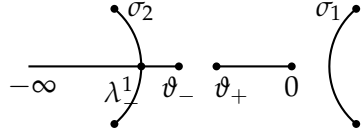
\begin{center}
\begin{tikzpicture}
\draw[thick] (-3.5,0)--(-1.5,0);
\draw[thick] (-1,0)--(0,0);
      \draw [thick,domain=-50:50] plot ({cos(\x)-3}, {sin(\x)});
         \draw [thick,domain=130:230] plot ({cos(\x)+1.5}, {sin(\x)});
  \fill ({cos(50)-3},{sin(50)}) circle (1.5pt);
    \fill ({cos(50)-3},{-sin(50)}) circle (1.5pt);
    \fill ({cos(130)+1.5},{sin(130)}) circle (1.5pt);
     \fill ({cos(130)+1.5},{-sin(130)}) circle (1.5pt);
   \fill (0,0) circle (1.5pt) node[below]{0};
      \fill (-1,0) circle (1.5pt);
      \draw(-0.8, 0) node[below]{$\vartheta_+$};
         \fill (-1.5,0) circle (1.5pt);
      \draw(-1.5, 0) node[below]{$\vartheta_-$};
      \fill (-2,0) circle (1.5pt);
      \draw(-2.25, 0) node[below]{$\l_-^1$};
       \draw(-3.45, 0) node[below]{$-\infty$};
         \draw(-2, 1) node[below]{$\sigma_2$};
           \draw(0.5, 1) node[below]{$\sigma_1$};
\end{tikzpicture}
\end{center}
where $\sigma_1, \sigma_2$ denote the spectral arcs  which do not lie on $\mathbb{R}$.  In particular, $\sigma_1, \sigma_2$ are symmetric with respect to $\mathbb{R}$ and disjoint with each other by Theorem \ref{thm-spec2}.
Similarly, we put the rough graphs of $\sigma_i(L^3_{b})=-\frac{1}{4b^2}\sigma(L^3_{\frac{1}{4b}})$ in blue dashed lines and $\sigma(L^3_{b})$ in dark solid lines together for $b\in (\widehat\beta-\delta, \widehat\beta)$.
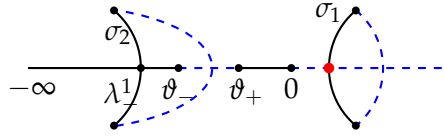
\begin{center}
\begin{tikzpicture}
\draw[thick] (-3.5,0)--(-1.5,0);
\draw[thick] (-0.7,0)--(0,0);
\draw[thick,dashed,blue] (2,0)--(0,0);
\draw[thick,dashed,blue] (-1.5,0)--(-0.7,0);
   \draw [thick,dashed,blue,domain=-70:70] plot ({1.95*cos(\x)-3}, {0.82*sin(\x)});
      \draw [thick,dashed,blue,domain=-52:52] plot ({cos(\x)+0.22}, {sin(\x)});
      \draw [thick,domain=-50:50] plot ({cos(\x)-3}, {sin(\x)});
         \draw [thick,domain=130:230] plot ({cos(\x)+1.5}, {sin(\x)});
  \fill ({cos(50)-3},{sin(50)}) circle (1.5pt);
    \fill ({cos(50)-3},{-sin(50)}) circle (1.5pt);
    \fill ({cos(130)+1.5},{sin(130)}) circle (1.5pt);
     \fill ({cos(130)+1.5},{-sin(130)}) circle (1.5pt);
   \fill (0,0) circle (1.5pt) node[below]{0};
      \fill (-0.7,0) circle (1.5pt);
      \draw(-0.6, 0) node[below]{$\vartheta_+$};
         \fill (-1.5,0) circle (1.5pt);
      \draw(-1.5, 0) node[below]{$\vartheta_-$};
      \fill (-2,0) circle (1.5pt);
      \fill[red] (0.5,0) circle (2pt);
      \draw(-2.25, 0) node[below]{$\l_-^1$};
       \draw(-3.45, 0) node[below]{$-\infty$};
         \draw(-2.3, 0.7) node[below]{$\sigma_2$};
           \draw(0.5, 1) node[below]{$\sigma_1$};
\end{tikzpicture}
\end{center}
From this graph, it is clear to see that the red point is the only element of $\Xi_{b}\cup \mathbb{R}$, so $\#(\Xi_{b}\cup \mathbb{R})=1$ for $b\in (\widehat\beta-\delta, \widehat\beta)$. By Corollary \ref{cor-even}, $k_1\in (0, \widehat\beta)$ and $k_2\in (\widehat\beta, \frac12)$, i.e., $\#(\Xi_{b}\cup \mathbb{R})=2$ for $b\in (0, k_1)$ and $\#(\Xi_{b}\cup \mathbb{R})=1$ for $b\in [k_1, k_2)$. Note that the topological graphs of $\sigma_i(L_b^3)$ are the same for all $b\in (0,\widehat{\beta})$, i.e., they are the blue dashed lines in the above picture. Since $d(0)=1$ for any $b\in (0,\widehat{\beta})$, the intersection of $\sigma_1$ and $\mathbb{R}$ can not pass through $0$, so  $\sigma_1\cap \mathbb{R}=\{r_+\}$ with $r_+>0$. Hence $r_+\in \Xi_b\cap \mathbb{R}$
for all $b\in (0, \widehat\beta)$. By the above figure, we know $\sigma_2$ will pass $\vartheta_-$ once and cannot pass through $\vartheta_+$. Denote by $\sigma_2\cap \mathbb{R}=\{r_-\}$, we have $r_-=\lambda_-^1<\vartheta_-$ for  $b\in (k_1, \widehat\beta)$ and $r_-=\lambda_-^1\in (\vartheta_-, \vartheta_+)$ for  $b\in (0,k_1)$. In particular, $r_-=\vartheta_-$ for  $b=k_1$, equivalently,  $d(\vartheta_-)>1$ for  $b=k_1$. Consequently, the rough graphs of $\sigma(L_b^3)$ with $b\in (0, \widehat\beta)$ are as follows.

\medskip
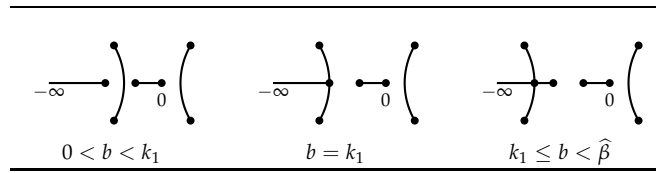
\begin{center}
\begin{tabular}{ccccc}
    \hline\\
\begin{tikzpicture}
\draw[thick] (-1.5,0)--(-0.75,0);
\draw[thick] (-0.35,0)--(0,0);
      \draw [thick,domain=-30:30] plot ({cos(\x)-1.5}, {sin(\x)});
         \draw [thick,domain=150:210] plot ({cos(\x)+1.25}, {sin(\x)});
  \fill ({cos(30)-1.5},{sin(30)}) circle (1.5pt) ;
   \fill ({cos(30)-1.5},{-sin(30)}) circle (1.5pt) ;
      \fill ({cos(150)+1.25},{sin(150)}) circle (1.5pt) ;
            \fill ({cos(150)+1.25},{-sin(150)}) circle (1.5pt) ;
   \fill (0,0) circle (1.5pt) node[below]{\tiny{0}};
      \fill (-0.35,0) circle (1.5pt);
 \fill(-0.75,0) circle (1.5pt);
       \draw(-1.5, 0.1) node[below]{\tiny{$-\infty$}};
\end{tikzpicture}
  & &
\begin{tikzpicture}
\draw[thick] (-1.5,0)--(-0.75,0);
\draw[thick] (-0.35,0)--(0,0);
      \draw [thick,domain=-30:30] plot ({cos(\x)-1.75}, {sin(\x)});
         \draw [thick,domain=150:210] plot ({cos(\x)+1.25}, {sin(\x)});
  \fill ({cos(30)-1.75},{sin(30)}) circle (1.5pt) ;
   \fill ({cos(30)-1.75},{-sin(30)}) circle (1.5pt) ;
      \fill ({cos(150)+1.25},{sin(150)}) circle (1.5pt) ;
            \fill ({cos(150)+1.25},{-sin(150)}) circle (1.5pt) ;
   \fill (0,0) circle (1.5pt) node[below]{\tiny{0}};
      \fill (-0.35,0) circle (1.5pt);
 \fill(-0.75,0) circle (1.5pt);
       \draw(-1.5, 0.1) node[below]{\tiny{$-\infty$}};
\end{tikzpicture}
  & &\begin{tikzpicture}
\draw[thick] (-1.5,0)--(-0.75,0);
\draw[thick] (-0.35,0)--(0,0);
      \draw [thick,domain=-30:30] plot ({cos(\x)-2}, {sin(\x)});
         \draw [thick,domain=150:210] plot ({cos(\x)+1.25}, {sin(\x)});
  \fill ({cos(30)-2},{sin(30)}) circle (1.5pt) ;
   \fill ({cos(30)-2},{-sin(30)}) circle (1.5pt) ;
      \fill ({cos(150)+1.25},{sin(150)}) circle (1.5pt) ;
            \fill ({cos(150)+1.25},{-sin(150)}) circle (1.5pt) ;
   \fill (0,0) circle (1.5pt) node[below]{\tiny{0}};
      \fill (-0.35,0) circle (1.5pt);
 \fill(-0.75,0) circle (1.5pt);
      \fill (-1,0) circle (1.5pt);
       \draw(-1.5, 0.1) node[below]{\tiny{$-\infty$}};
\end{tikzpicture}
\\
\scriptsize{ $0<b<k_1$}&& \scriptsize{ $b=k_1$} && \scriptsize{ $k_1\leq b<\widehat\beta$}\\
\hline
    \end{tabular}
    \end{center}
    \medskip

Before we end this step, note that $d(0)>1$ if and only if $b=b_g$, we get the following conclusions.

\begin{proposition}
    \label{prop-sigma-1}
For any $b\in (0, b_g)$, the spectrum $\sigma(L_b^3)$  always includes a simple spectral arc denoted by $\sigma_1$ which is symmetric with respect to $\mathbb{R}$ and $\sigma_1\cap\mathbb{R}=\{r_+\}$ with $r_+>0$.
\end{proposition}

Furthermore, combining Proposition \ref{prop-neginterval}, we obtain that $k_2=b_g$.


\medskip

    \textbf{Step4: The spectrum for $b<\beta$ and close to $\beta$. }

First, if $b=\beta$, then $\Delta_{R_1}({\beta})=0$, i.e., $5g_2=12e_1^2$, thus $5g_3=8e_1^3$. A direct computation gives us $f(\l_L)f(\l_R)\vert_{b={\beta}}<0$.  For $b\in (\beta-\delta,\beta)$ with $\delta>0$ sufficient small, we have $f(\l_L)f(\l_R)<0$, then all roots of $f$ are real, i.e., all inner intersection points of $\sigma(L_b^3)$ are real.
By Theorem \ref{thm-spec-poly-zeros}, we have $\vartheta_-=\overline{\vartheta_+}\notin\mathbb{R}$. Similar to $b>\beta$ case, we know the rough graph of $\sigma(L_b^3)$  consists of $\Sigma, \sigma$ and $\sigma'$, where $\sigma,\sigma'$ are two simple spectral arcs and $\sigma\cup \sigma'$ are symmetric with respect to $\mathbb{R}$. Furthermore, we have $(\sigma \cup \sigma')\cap \Sigma=\emptyset$ by the rough graph of $\sigma(L_\beta^3)$. Draw the rough graphs of   $\sigma_i(L^3_{b})=-\frac{1}{4b^2}\sigma(L^3_{\frac{1}{4b}})$ using blue dashed lines and $\Sigma$ using dark solid lines in the following:
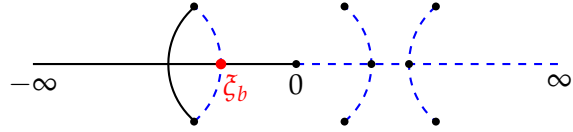
\begin{center}
\begin{tikzpicture}
\draw[thick,dashed, blue] (-2.5,0)--(1,0);
\draw[thick] (-6,0)--(-2.5,0);
   \draw [thick,dashed, blue,domain=130:230] plot ({cos(\x)}, {sin(\x)});
    \draw [thick,domain=130:230] plot ({cos(\x)-3.2}, {sin(\x)});
      \draw [thick,dashed, blue,domain=-50:50] plot ({cos(\x)-2.5}, {sin(\x)});
         \draw [thick,dashed, blue,domain=-50:50] plot ({cos(\x)-4.5}, {sin(\x)});
  \fill ({cos(50)-2.5},{sin(50)}) circle (1.5pt); 
    \fill ({cos(50)-2.5},{-sin(50)}) circle (1.5pt); 
    \fill ({cos(130)},{sin(50)}) circle (1.5pt);
      \fill ({cos(130)},{-sin(50)}) circle (1.5pt);
    \fill ({cos(50)-4.5},{-sin(50)}) circle (1.5pt);
    \fill ({cos(50)-4.5},{sin(50)}) circle (1.5pt);
    \fill[red](-3.5,0) circle (2pt);
       \draw[red](-3.3,0) node[below]{$\xi_b$};
   \fill (-2.5,0) circle (1.5pt) node[below]{0};
      \fill (-1,0) circle (1.5pt);
         \fill (-1.5,0) circle (1.5pt);
          \draw(-6, 0) node[below]{$-\infty$};
     \draw(1, 0) node[below]{$\infty$};
\end{tikzpicture}
\end{center}
It is clear to see that $\xi_b\in \Xi_b\cap \mathbb{R}$.
Since $\frac{1}{4k_1}>\beta$, we have  $\#(\Xi_{b}\cap \mathbb{R})\leq 1$, 
thus $\Xi_b\cap \mathbb{R}=\{\xi_b\}\subseteq \Sigma$. So $(\sigma_1\cup \sigma_2)\cap \mathbb{R}=\emptyset$. 
Note that all inner intersection points of $\sigma(L_b^3)$ are real. 
Therefore, the rough graph of $\sigma(L_b^3)$ with $b\in (\beta-\delta,\beta)$ is the following:
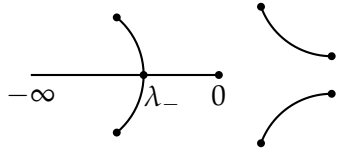
\begin{center}
\begin{tikzpicture}
\draw[thick] (-2.5,0)--(0,0);
      \draw [thick,domain=-50:50] plot ({cos(\x)-2}, {sin(\x)});
         \draw [thick,domain=200:270] plot ({cos(\x)+1.5}, {sin(\x)+1.25});
            \draw [thick,domain=90:160] plot ({cos(\x)+1.5}, {sin(\x)-1.25});
  \fill ({cos(50)-2},{sin(50)}) circle (1.5pt);
    \fill ({cos(50)-2},{-sin(50)}) circle (1.5pt);
    \fill ({cos(200)+1.5},{sin(200)+1.25}) circle (1.5pt);
     \fill ({cos(270)+1.5},{sin(270)+1.25}) circle (1.5pt);
         \fill ({cos(90)+1.5},{sin(90)-1.25}) circle (1.5pt);
     \fill ({cos(160)+1.5},{sin(160)-1.25}) circle (1.5pt);
   \fill (0,0) circle (1.5pt) node[below]{0};
         \fill (-1,0) circle (1.5pt);
      \draw(-0.75, 0) node[below]{$\l_-$};
      \draw(-2.5, 0) node[below]{$-\infty$};
\end{tikzpicture}
\end{center}

The proof is complete.

\section*{Acknowledgments}
The research of
the author is supported by BIMSA Start-up Research Fund.

\end{proof}

\end{document}